\documentclass[leqno, 12pt]{article}
\usepackage{amsmath,amsfonts,amsthm,amssymb,indentfirst}

\newtheorem{theorem}{Theorem}
\newtheorem{lemma}[theorem]{Lemma}
\newtheorem{definition}[theorem]{Definition}
\newtheorem{corollary}[theorem]{Corollary}
\newtheorem{proposition}[theorem]{Proposition}

\theoremstyle{definition}

\newcommand{\Sym}{\mathrm{Sym}}

\newcommand{\Inj}{\mathrm{Inj}}

\newcommand{\C}{\mathrm{C}}

\newcommand{\Z}{\mathbb{Z}}
\newcommand{\N}{\mathbb{N}}

\begin{document}

\title{Conjugation of injections by permutations}

\author{Zachary Mesyan\thanks{This work was done while the author was supported by a Postdoctoral Fellowship from the Center for Advanced Studies in Mathematics at Ben Gurion University, a Vatat Fellowship from the Israeli Council for Higher Education, and ISF grant 888/07.}}

\maketitle

\begin{abstract}
Let $\Omega$ be a countably infinite set, $\Inj(\Omega)$ the monoid of all
injective endomaps of $\Omega$, and $\Sym(\Omega)$ the group of all permutations of $\Omega$. Also, let $f, g, h \in \Inj(\Omega)$ be any three maps, each having at least one infinite cycle. (For instance, this holds if $f,g,h \in \Inj(\Omega) \setminus \Sym(\Omega)$.) We show that there are permutations $a, b \in \Sym(\Omega)$ such that $h=afa^{-1}bgb^{-1} $ if and only if $|\Omega \setminus (\Omega)f| + |\Omega \setminus (\Omega)g| = |\Omega \setminus (\Omega)h|$. We also prove a generalization of this statement that holds for infinite sets $\Omega$ that are not necessarily countable. 
\end{abstract}

\section{Introduction}
Let $\Omega$ be a countably infinite set, $\Inj(\Omega)$ the
monoid of all injective endomaps of $\Omega$, and $\Sym(\Omega)$
the group of all permutations of $\Omega.$ In this paper we investigate the following question: given two maps $f,g \in \Inj(\Omega)$, which elements $h \in \Inj(\Omega)$ can be expressed in the form $h = afa^{-1}bgb^{-1}$ ($a, b \in \Sym(\Omega)$)?

In the case where $f$ and $g$ (and hence also $h$) are bijections, the above question has been studied extensively in the literature. In particular, Moran described all the permutations $f \in \Sym(\Omega)$ such that every element of $\Sym(\Omega)$ can be expressed as a product of two conjugates of $f$. This description relies on many earlier results of Droste, Bertram, Gray, and others (see~\cite{Moran}, \cite{Droste}, and the work referenced therein). Another important result in this vein, due to Droste~\cite{Droste2}, is that if $f,g \in \Sym(\Omega)$ are permutations, both having at least one infinite cycle, then every permutation $h \in \Sym(\Omega)$ that has infinite support can be expressed in the form $h = afa^{-1}bgb^{-1}$, for some $a, b \in \Sym(\Omega)$. (We recall that a map $f \in \Sym(\Omega)$, or more generally, $f \in \Inj(\Omega)$ has an infinite cycle if there exists $\alpha \in \Omega$ such that $\alpha \neq (\alpha)f^i$ for all positive integers $i$. Also, the support of a map $f \in \Inj(\Omega)$ is the set of all elements of $\Omega$ that are moved by $f$.)

Our goal here is to extend Droste's theorem to maps $f, g, h \in \Inj(\Omega)$ that are not necessarily permutations. More specifically we prove that if $f, g, h \in \Inj(\Omega)$ are any three maps having infinite support, such that $f$ and $g$ both have at least one infinite cycle, then $h=afa^{-1}bgb^{-1}$ for some $a, b \in \Sym(\Omega)$ if and only if $|\Omega \setminus (\Omega)f| + |\Omega \setminus (\Omega)g| = |\Omega \setminus (\Omega)h|$. (If $f,g,h \in \Sym(\Omega)$, then the condition $|\Omega \setminus (\Omega)f| + |\Omega \setminus (\Omega)g| = |\Omega \setminus (\Omega)h|$ is satisfied trivially.) It is easy to see that every injective endomap of $\Omega$ that is not a bijection has at least one infinite cycle (and hence also infinite support). Therefore, our result completely answers the question posed in the first paragraph above when $f,g \in \Inj(\Omega) \setminus \Sym(\Omega)$.  

One can quickly obtain a version of our main result that holds for arbitrary infinite sets $\Omega.$ Namely, let $f, g, h \in \Inj(\Omega)$ be any three maps having support of cardinality $K$ ($\aleph_0 \leq K \leq |\Omega|$), such that $K$-many elements of $\Omega$ belong to infinite cycles of both $f$ and $g$. Then there exist permutations $a, b \in \Sym(\Omega)$ such that $h = afa^{-1}bgb^{-1}$ if and only if $|\Omega \setminus (\Omega)f| + |\Omega \setminus (\Omega)g| = |\Omega \setminus (\Omega)h|$. The main result is also used in~\cite{ZM} to describe all the submonoids of $\Inj(\Omega)$ that are closed under conjugation by elements of $\Sym(\Omega)$.

Incidentally, initial interest in the question posed in the first paragraph above came from the result of Ore~\cite{Ore} that every permutation of an infinite set $\Omega$ can be expressed as a commutator, equivalently, that every $h \in\Sym(\Omega)$ can be expressed in the form $h=afa^{-1}bfb^{-1}$ for some $a, b, f \in \Sym(\Omega)$. As a consequence of our main result, we obtain the following generalization of Ore's Theorem: an element $h \in\Inj(\Omega)$ can be expressed in the form $h = afa^{-1}bfb^{-1}$, for some $f \in \Inj(\Omega)$ and $a, b \in \Sym(\Omega)$, if and only if $\, |\Omega \setminus (\Omega)h|$ is either an even integer or infinite.

\subsection*{Acknowledgement}  

The author is grateful to George Bergman for comments on an earlier draft of this note, which have led to significant improvements, and to the referee for suggesting further questions to explore.

\section{Conjugation basics}

We begin by defining notation and making a few basic observations that will be used throughout. Let $\Omega$ be an arbitrary infinite set,
$\Inj(\Omega)$ the monoid of all injective endomaps of $\Omega$,
and $\Sym(\Omega)$ the group of all permutations of $\Omega.$ We
shall write set maps on the right of their arguments. The set of
integers will be denoted by $\Z,$ the set of positive integers
will be denoted by $\Z_+,$ the set of nonnegative integers will be
denoted by $\N,$ and the cardinality of a set $\Sigma$ will be
denoted by $|\Sigma|$.

\begin{definition}
Let $f \in \Inj(\Omega)$ be any element, and let $\, \Sigma \subseteq
\Omega$ be a nonempty set. We shall say that $\, \Sigma$ is a {\em
cycle} under $f$ if the following two conditions are satisfied:
\begin{enumerate}
\item[$(${\rm i}$)$] for all $\alpha \in \Omega,$ $(\alpha)f
\in \Sigma$ if and only if $\alpha \in \Sigma;$
\item[$(${\rm ii}$)$] $\Sigma$ has no proper nonempty subset satisfying
 $(${\rm i}$)$.
\end{enumerate}
We shall say that $\, \Sigma$ is a {\em forward cycle} under $f$
if $\, \Sigma$ is an infinite cycle under $f$ and there is an
element $\alpha \in \Sigma$ such that for all $\beta \in \Omega,$
$(\beta)f \neq \alpha.$ In this case, we shall refer to $\alpha$ as the {\em initial element} of $\, \Sigma$.

If $\, \Sigma$ is an infinite cycle under $f$ that is not a forward cycle, we shall refer to it as an {\em open cycle}.
\end{definition}

It is easy to see that for any $\alpha \in \Omega,$ the set
$$\{(\alpha)f^n : n \in \N\}\cup \{\beta \in \Omega : \exists n \in \Z_+ \,
((\beta)f^n = \alpha)\}$$ is a cycle under $f.$ By condition (ii)
above, it follows that every cycle of $f$ is of this form. This
also implies that every $\alpha \in \Omega$ falls into
exactly one cycle under $f,$ and that there can only be one initial element in a forward cycle of $f$, justifying the usage of the phrase ``{\em the} initial element" in the definition. Thus, we can define a collection
$\{\Sigma_i\}_{i\in I}$ of disjoint subsets of $\Omega$ to be a
{\em cycle decomposition} of $f$ if each $\Sigma_i$ is a cycle
under $f$ and $\bigcup_{i\in I} \Sigma_i = \Omega.$ We note that
$f$ can have only one cycle decomposition, up to reindexing the
cycles. For convenience, we shall therefore at times refer to {\em
the} cycle decomposition of $f.$

\begin{definition}
Let $f,g \in \Inj(\Omega)$ be any two elements. We shall say that
$f$ and $g$ have {\em equivalent} cycle decompositions if there
exists an indexing set $I$ and cycle decompositions $\,
\{\Sigma_i\}_{i\in I}$ and $\, \{\Gamma_i\}_{i\in I}$ of $f$ and
$g,$ respectively, that satisfy the following two conditions:
\begin{enumerate}
\item[$(${\rm i}$)$] for all $i \in I,$ $|\Sigma_i| = |\Gamma_i|;$
\item[$(${\rm ii}$)$] if $\, |\Sigma_i| = |\Gamma_i| = \aleph_0$ for some $i \in I,$ then $\, \Sigma_i$ is a forward cycle under $f$ if and only if $\,
\Gamma_i$ is a forward cycle under $g.$
\end{enumerate}
In this case, we shall write $f \sim g.$
\end{definition}

It is clear that $\sim$ is an equivalence relation on $\Inj(\Omega)$. As in the case of permutations, two injective endomaps $f$ and $g$
have equivalent cycle decompositions if and only if they are
conjugate to each other. For completeness, we provide a proof of
this, which is analogous to the one for permutations.

\begin{proposition}\label{decomposition}
Let $f,g \in \Inj(\Omega)$ be any two maps. Then $g = afa^{-1}$
for some $a \in \Sym(\Omega)$ if and only if $f \sim g$.
\end{proposition}

\begin{proof}
Suppose that $g = afa^{-1}$ for some $a \in \Sym(\Omega),$ and let
$\{\Sigma_i\}_{i\in I}$ be a cycle decomposition for $f$. It is
easy to check that $\{(\Sigma_i)a^{-1}\}_{i\in I}$ is a cycle
decomposition for $g,$ and that $\{\Sigma_i\}_{i\in I}$ and
$\{(\Sigma_i)a^{-1}\}_{i\in I}$ are equivalent.

Conversely, let $\{\Sigma_i\}_{i\in I}$ and $\{\Gamma_i\}_{i\in
I}$ be equivalent cycle decompositions for $f$ and $g,$
respectively. For each $\Sigma_i$ we shall define a set of integers
$J_i.$  If $\Sigma_i$ is finite, then let $J_i = \{0, 1, \dots,
n-1\}$, where $n = |\Sigma_i|;$ if $\Sigma_i$ is a forward cycle,
then let $J_i = \N;$ if $\Sigma_i$ is an open cycle, then let $J_i = \Z.$
We can then write each $\Sigma_i$ as $\Sigma_i =
\{\sigma_{i,j}\}_{j\in J_i},$ where for each $j \in J_i,$ we have
$(\sigma_{i,j})f = \sigma_{i,j+1}$ (or, $(\sigma_{i,j})f =
\sigma_{i,j+1} \, (\mathrm{mod} \ |\Sigma_i|),$ if $\Sigma_i$ is
finite). We can similarly write each $\Gamma_i$ as $\Gamma_i =
\{\gamma_{i,j}\}_{j\in J_i}.$ We next define an element $a \in
\Sym(\Omega)$ by $(\gamma_{i,j})a = \sigma_{i,j}$ for all $i$ and
$j$. (Since $\dot{\bigcup}_{i\in I} \Sigma_i = \Omega$ and
$\dot{\bigcup}_{i\in I} \Gamma_i = \Omega,$ this indeed defines an
element of $\Sym(\Omega).$) Then for all $i$ and $j$ we have
$$(\gamma_{i,j})afa^{-1} = (\sigma_{i,j})fa^{-1} =
(\sigma_{i,j+1})a^{-1} = \gamma_{i,j+1} = (\gamma_{i,j})g$$ (where
the equalities are modulo $|\Sigma_i| = |\Gamma_i|,$ whenever
appropriate). Hence $afa^{-1} = g.$
\end{proof}

The following two observations, while easy to prove, will be important in the sequel.

\begin{lemma}\label{fwd cycles}
Let $f \in \Inj(\Omega)$ be any map. Then there is a one-to-one
correspondence between the elements of $\, \Omega \setminus
(\Omega)f$ and forward cycles in the cycle decomposition of $f$.
\end{lemma}

\begin{proof}
Suppose that $\alpha \in \Omega \setminus (\Omega)f.$ Then
$\{(\alpha)f^k : k \in \N\}$ is a forward cycle under $f.$ (Since
$f$ is injective and $\alpha \in \Omega \setminus (\Omega)f,$
$(\alpha)f^i \neq (\alpha)f^j$ for $i \neq j.$) Thus, each element
of $\Omega \setminus (\Omega)f$ is contained in a forward cycle in
the decomposition of $f.$ On the other hand, by the definition of
``forward cycle," each such cycle contains exactly one element of
$\Omega \setminus (\Omega)f.$ 
\end{proof}

\begin{lemma}\label{coimage}
Let $g \in \Inj(\Omega)$ be a map, and let $\, \Delta \subseteq \Omega$ be a subset. Then $\, |\Omega \setminus \Delta| + |\Omega \setminus (\Omega)g| = |\Omega \setminus (\Delta)g|.$ In particular, for all $f, g \in \Inj(\Omega)$ we have $\, |\Omega \setminus (\Omega)f| + |\Omega \setminus (\Omega)g| = |\Omega \setminus (\Omega)fg|.$
\end{lemma}

\begin{proof}
Set $\Omega \setminus \Delta = \Sigma$ and $\Omega \setminus
(\Omega)g = \Gamma.$ Then, in particular, $(\Sigma)g \cap \Gamma =
\emptyset.$ Since $(\Delta)g \subseteq (\Omega)g,$ we have
$\Omega \setminus (\Delta)g \supseteq \Omega \setminus (\Omega)g
= \Gamma.$ Also, $\Omega \setminus (\Delta)g \supseteq
(\Sigma)g,$ since otherwise there would exist elements $\alpha \in
\Sigma$ and $\beta \in \Delta$ such that $(\beta)g = (\alpha)g,$
contradicting the injectivity of $g.$ Therefore, $\Omega \setminus
(\Delta)g \supseteq (\Sigma)g \cup \Gamma.$ Further, if $\alpha
\in (\Omega)g$ but $\alpha \notin (\Delta)g,$ then $\alpha$ must
be an element of $(\Sigma)g.$ Thus, $\Omega \setminus (\Delta)g
\subseteq (\Sigma)g \cup \Gamma.$ The desired statement then
follows from the fact that $(\Sigma)g \, \dot{\cup} \, \Gamma = \Omega \setminus (\Delta)g.$ The final claim is immediate.
\end{proof}

\begin{definition}\label{cycle number}
For each $f \in \Inj(\Omega)$ and $n \in \Z_+$ let $$(f)\C_n =
|\{\Sigma \subseteq \Omega : \Sigma \ \mathrm{is} \ \mathrm{a} \
\mathrm{cycle} \ \mathrm{under} \ f \ \mathrm{of} \
\mathrm{cardinality} \ n\}|.$$ Similarly, let
$$(f)\C_{\mathrm{open}} = |\{\Sigma \subseteq \Omega : \Sigma \
\mathrm{is} \ \mathrm{an} \ \mathrm{open} \ \mathrm{cycle} \
\mathrm{under} \ f\}|$$ and
$$(f)\C_{\mathrm{fwd}} = |\{\Sigma \subseteq \Omega :  \Sigma \ \mathrm{is} \ \mathrm{a} \ \mathrm{forward} \ \mathrm{cycle} \ \mathrm{under} \
f\}|.$$
\end{definition}

Thus, given two elements $f,g \in \Inj(\Omega)$, we have $f \sim g$ if and only if $(f)\C_{\mathrm{open}} = (g)\C_{\mathrm{open}},$
$(f)\C_{\mathrm{fwd}} = (g)\C_{\mathrm{fwd}},$ and $(f)\C_n
= (g)\C_n$ for all $n \in \Z_+.$

\begin{lemma} \label{pass}
Let $\, \Omega$ and $\, \Omega'$ be two infinite sets such that $\, |\Omega| = |\Omega'|$. Also, let $f \in \Inj(\Omega)$ and $f' \in \Inj(\Omega')$ be maps such that $(f)\C_{\mathrm{open}} = (f')\C_{\mathrm{open}},$
$(f)\C_{\mathrm{fwd}} = (f')\C_{\mathrm{fwd}},$ and $(f)\C_n
= (f')\C_n$ for all $n \in \Z_+.$ Then, for any bijection $g : \Omega' \rightarrow \Omega$ we have $f' \sim gfg^{-1}$.
\end{lemma}

\begin{proof}
Let $\{\Sigma_i\}_{i\in I}$ be a cycle decomposition for $f$. Then $\{(\Sigma_i)g^{-1}\}_{i\in I}$ can be easily checked to be a cycle decomposition for $gfg^{-1}$. It follows that $(f)\C_{\mathrm{open}} = (gfg^{-1})\C_{\mathrm{open}},$
$(f)\C_{\mathrm{fwd}} = (gfg^{-1})\C_{\mathrm{fwd}},$ and $(f)\C_n
= (gfg^{-1})\C_n$ for all $n \in \Z_+$. Hence, by the hypotheses and the above remark, we have $f' \sim gfg^{-1}$.
\end{proof}

Before concluding this section, let us give one more easy observation, which will be convenient to have on record.

\begin{lemma} \label{prod-conj-equiv}
Let $f, g, h \in \Inj(\Omega)$ be any three maps. Then the following conditions are equivalent.
\begin{enumerate}
\item[$(1)$] There exist permutations $a, b \in \Sym(\Omega)$ such that
$h=afa^{-1}bgb^{-1}$.
\item[$(2)$] There exist elements $f',g' \in \Inj(\Omega)$ such that $f' \sim f$, $g' \sim g$, and $f'g' \sim h$.
\end{enumerate}
\end{lemma}

\begin{proof}
If $h=afa^{-1}bgb^{-1}$ for some $a, b \in \Sym(\Omega)$, then $h = f'g'$, where $f' = afa^{-1} \sim f$ and $g' = bgb^{-1} \sim g$, by Proposition~\ref{decomposition}. Hence (1) implies (2). 

To prove the converse, assume that we have $f',g' \in \Inj(\Omega)$ as in (2). Then, by Proposition~\ref{decomposition}, we can find permutations $c, d, e \in \Sym(\Omega)$ such that  $e^{-1}he = cfc^{-1}dgd^{-1}$. It follows that $h = e(cfc^{-1}dgd^{-1})e^{-1} = (ec)f(ec)^{-1}(ed)g(ed)^{-1},$ as desired.
\end{proof}

\section{A composition theorem}

We turn to our main result.

\begin{theorem}\label{comp}
Let $\, \Omega$ be a countably infinite set, and let $f, g, h \in
\Inj(\Omega)$ be any three maps, each having at least one infinite cycle. Then there exist permutations $a, b \in \Sym(\Omega)$ such that
$h=afa^{-1}bgb^{-1}$ if and only if $\, |\Omega \setminus
(\Omega)f| + |\Omega \setminus (\Omega)g| = |\Omega \setminus
(\Omega)h|.$
\end{theorem}

Before giving the proof, let us discuss certain features and consequences of this theorem. First, we note that, by Lemma~\ref{fwd cycles}, every element of $\Inj(\Omega)$ that is not a bijection has at least one forward (and hence infinite) cycle. Therefore, we have the following special case.

\begin{corollary}
Let $\, \Omega$ be a countably infinite set, and let $f, g, h \in
\Inj(\Omega) \setminus \Sym(\Omega)$ be any three maps. Then there exist permutations $a, b \in \Sym(\Omega)$ such that
$h=afa^{-1}bgb^{-1}$ if and only if $\, |\Omega \setminus
(\Omega)f| + |\Omega \setminus (\Omega)g| = |\Omega \setminus
(\Omega)h|.$
\end{corollary}

While proving Theorem~\ref{comp}, we shall assume that $f$ and $g$ are not both permutations, since in that case our result is a consequence of the following theorem from~\cite{Droste2}. (Let us recall that the {\em support} of a map $f \in \Inj(\Omega)$ is the set $\{\alpha \in \Omega : (\alpha)f \neq \alpha\}$.)

\begin{theorem}[Droste] \label{droste theorem}
Let $\, \Omega$ be a countably infinite set, and let $f,g,h \in \Sym(\Omega)$ be any three permutations that have infinite support. If both $f$ and $g$ have at least one infinite cycle, then $h = afa^{-1}bgb^{-1}$ for some $a, b \in \Sym(\Omega)$.
\end{theorem}

Actually, again in view of Lemma~\ref{fwd cycles}, Theorem~\ref{droste theorem} allows us to state our main result in the following more general (though also more complicated) way.

\begin{corollary} \label{general comp}
Let $\, \Omega$ be a countably infinite set, and let $f, g, h \in \Inj(\Omega)$ be any three maps that have infinite support. Further, suppose that $f$ and $g$ both have at least one infinite cycle. Then there exist permutations $a, b \in \Sym(\Omega)$ such that $h=afa^{-1}bgb^{-1}$ if and only if $\, |\Omega \setminus (\Omega)f| + |\Omega \setminus (\Omega)g| = |\Omega \setminus
(\Omega)h|.$
\end{corollary}

Thus, we may weaken the hypothesis in Theorem~\ref{comp} that $h$ must have an infinite cycle. However, some assumptions on the cycle decompositions of $f$, $g$, and $h$, along the above lines, are necessary for the statement in Theorem~\ref{comp} (and Corollary~\ref{general comp}) to hold. For example, let $f, h \in \Inj(\Omega)$ be maps such that $|\Omega \setminus (\Omega)f| = |\Omega \setminus (\Omega)h|$, and so that $f$ fixes infinitely many elements of $\Omega$ but $h$ does not. Also, let $g \in \Sym(\Omega)$ be any element having finite support. Then $|\Omega \setminus (\Omega)f| + |\Omega \setminus (\Omega)g| = |\Omega \setminus (\Omega)h|$, but $h \neq afa^{-1}bgb^{-1}$ for any $a, b \in \Sym(\Omega),$ since $afa^{-1}bgb^{-1}$ must fix infinitely many elements of $\Omega.$ It is also well known \cite[Theorem 2.2]{Bertram} that taking $f = g$ to be a permutation that has one (open) infinite cycle and no other cycles, one can find an element $h \in \Sym(\Omega)$ (having finite support) that cannot be expressed in the form $afa^{-1}bgb^{-1}.$ Another well-known fact \cite[Theorem 1]{Moran} is that taking $f = g$ to be a permutation that has no cycles of cardinality greater than $2$, one can find an element $h \in \Sym(\Omega)$ that cannot be expressed in the form $afa^{-1}bgb^{-1}$.

Applying a standard argument to Corollary~\ref{general comp}, we can further generalize Theorem~\ref{comp} to arbitrary infinite sets $\Omega$. To make this generalization easier to state, let us introduce one more piece of notation. Given a map $f \in \Inj(\Omega)$, let $\Upsilon_f \subseteq \Omega$ denote the union of all the infinite cycles of $f$.

\begin{corollary} \label{uncountable}
Let $\, \Omega$ be an infinite set, and let $f, g, h \in \Inj(\Omega)$ be any three maps that have support of cardinality $K$, where $\, \aleph_0 \leq K \leq |\Omega|$. Further, suppose that $\, |\Upsilon_f| = K = |\Upsilon_g|$. Then there exist permutations $a, b \in \Sym(\Omega)$ such that $h=afa^{-1}bgb^{-1}$ if and only if $\, |\Omega \setminus (\Omega)f| + |\Omega \setminus (\Omega)g| = |\Omega \setminus (\Omega)h|.$
\end{corollary}

\begin{proof}
If $h=afa^{-1}bgb^{-1}$ for some $a, b \in \Sym(\Omega)$, then, by  Lemma~\ref{coimage}, we have $|\Omega \setminus (\Omega)f| + |\Omega \setminus (\Omega)g| = |\Omega \setminus (\Omega)h|$. (For,
$|\Omega \setminus (\Omega)f| = |\Omega \setminus (\Omega)afa^{-1}|$ and $|\Omega \setminus (\Omega)g| = |\Omega \setminus (\Omega)aga^{-1}|$, by Proposition~\ref{decomposition} and Lemma~\ref{fwd cycles}.)

For the converse, let $\Sigma \subseteq \Omega$ be a subset of cardinality $K$ which contains the support of $h$ (if $K = |\Omega|$, then we take $\Sigma = \Omega$). Then, in particular, $(\Sigma)h \subseteq \Sigma$, since $h$ is injective. Upon replacing $f$ and $g$ with conjugates, we may assume that their supports are contained in $\Sigma$ as well. If $K = \aleph_0$, then, viewing $f$, $g$, and $h$ as elements of $\Inj(\Sigma)$, the desired conclusion follows from Corollary~\ref{general comp}. Therefore, let us assume that $\aleph_0 < K$ and write $\Sigma = \bigcup_{i \in K} \Omega_i$, where the union is disjoint, and each $\Omega_i$ has cardinality $\aleph_0$. It follows from our hypotheses that $f$ and $g$ must both have $K$-many infinite cycles. We can therefore find maps $\bar{f},\bar{g},\bar{h} \in \Inj(\Omega)$ whose supports are contained in $\Sigma$, such that for all $i \in K$ the following conditions are satisfied:
\begin{enumerate}
\item[(1)] $f \sim \bar{f}$, $g \sim \bar{g}$, and $h \sim \bar{h}$;
\item[(2)] each of $\bar{f}$, $\bar{g}$, and $\bar{h}$ takes $\Omega_i$ to itself, moving infinitely many elements of $\Omega_i$;
\item[(3)] $\bar{f}$ and $\bar{g}$ each have at least one infinite cycle consisting of elements of $\Omega_i$;
\item[(4)] $|\Omega \setminus (\Omega)\bar{f}_i| + |\Omega \setminus (\Omega)\bar{g}_i| = |\Omega \setminus (\Omega)\bar{h}_i|$, where $\bar{f}_i,$ $\bar{g}_i,$ and $\bar{h}_i$ denote the restrictions of $\bar{f},$ $\bar{g},$ and $\bar{h}$, respectively, to $\Omega_i$.
\end{enumerate}
By Corollary~\ref{general comp}, for each $i \in K$ we can find permutations $a_i, b_i \in \Sym(\Omega_i)$ such that $\bar{h}_i = a_i\bar{f}_ia^{-1}_i b_i\bar{g}_ib^{-1}_i.$ Letting $a, b \in \Sym(\Omega)$ be permutations that act as the identity on $\Omega \setminus \Sigma$, such that the restriction of $a,$ respectively $b,$ to each $\Omega_i$ is $a_i,$ respectively $b_i,$ we have $\bar{h} = a\bar{f}a^{-1} b\bar{g}b^{-1}$. The desired conclusion now follows from Lemma~\ref{prod-conj-equiv} (and Proposition~\ref{decomposition}).
\end{proof}

The assumption in the above corollary that the three maps have support of equal size is necessary. We can see this from an example similar to one given above. Let $\Omega$ be an uncountable set, and let $f, g \in  \Inj(\Omega)$ be maps that both fix all but countably many elements of $\Omega$.  Then any conjugate of either map will likewise fix all but countably many elements of $\Omega$, as will any product of such conjugates.  Hence, if we take $h \in  \Inj(\Omega)$ to be a map that moves uncountably many elements of $\Omega$ (and satisfies $|\Omega \setminus (\Omega)f| + |\Omega \setminus (\Omega)g| = |\Omega \setminus (\Omega)h|$), then it would be impossible to represent it as a product of conjugates of $f$ and $g$.

As mentioned in the Introduction, perhaps the earliest result on the subject under discussion is the following.

\begin{theorem}[Ore]
Let $\, \Omega$ be an infinite set. Then every element of $\, \Sym(\Omega)$ can be expressed as a commutator; that is, in the form $afa^{-1}f^{-1}$ for some $a, f \in \Sym(\Omega)$. Equivalently, every $h \in\Sym(\Omega)$ can be expressed in the form $h = afa^{-1}bfb^{-1}$ for some $a, b, f \in \Sym(\Omega)$.
\end{theorem}

Using our results above, we can generalize this as follows.

\begin{corollary}
Let $\, \Omega$ be an infinite set. Then an element $h \in \Inj(\Omega)$ can be expressed in the form $h = afa^{-1}bfb^{-1}$, for some $f \in \Inj(\Omega)$ and $a, b \in \Sym(\Omega)$, if and only if $\, |\Omega \setminus (\Omega)h|$ is either an even integer or infinite.
\end{corollary}

\begin{proof}
As in the proof of Corollary~\ref{uncountable}, if $h=afa^{-1}bfb^{-1}$ for some $f \in \Inj(\Omega)$ and $a, b \in \Sym(\Omega)$, then $|\Omega \setminus (\Omega)f| + |\Omega \setminus (\Omega)f| = |\Omega \setminus (\Omega)h|$, from which the ``only if" direction follows. 

For the converse, let $h \in \Inj(\Omega)$ be such that $|\Omega \setminus (\Omega)h|$ is either an even integer or infinite. If $|\Omega \setminus (\Omega)h| = 0$, and hence $h \in \Sym(\Omega)$, then the desired conclusion follows from the above theorem of Ore. Therefore, suppose that $|\Omega \setminus (\Omega)h| \neq 0$, so that the cardinality $K$ of the support of $h$ is infinite (by Lemma~\ref{fwd cycles}). Let $f \in \Inj(\Omega)$ be any map with support of cardinality $K$, $|\Upsilon_f| = K$, and $|\Omega \setminus (\Omega)f| + |\Omega \setminus (\Omega)f| = |\Omega \setminus (\Omega)h|$ (which amounts to $|\Omega \setminus (\Omega)f| = |\Omega \setminus (\Omega)h|$, if $|\Omega \setminus (\Omega)h|$ is infinite). We then obtain the desired result from Corollary~\ref{uncountable}.
\end{proof}

Let us conclude this section by giving another perspective on the results mentioned here. For all $f \in \Inj(\Omega)$, let $f^S = \{afa^{-1} : a \in S\}$, where $S = \Sym(\Omega)$. Now, for any $f, g \in \Inj(\Omega)$, we have $f^Sg^S \subseteq (SfS)(SgS) = SfgS$, since for any $h \in \Inj(\Omega)$, $ShS = \{h' \in \Inj(\Omega) : |\Omega \setminus (\Omega)h| = |\Omega \setminus (\Omega)h'|\}$. That is, $f^Sg^S \subseteq SfgS$. The results in this section can be viewed as saying that this inclusion is actually an equality for most choices of $f$ and $g$.

\section{The plan} \label{plan section}

The remainder of the paper is devoted to the proof of
Theorem~\ref{comp}. The ``only if" direction follows (even without
any restrictions on $f, g, h \in \Inj(\Omega)$) from
Lemma~\ref{coimage}, with the help of Proposition~\ref{decomposition} and Lemma~\ref{fwd cycles}, as we have already seen in the proof of Corollary~\ref{uncountable}. Thus,
we only need to prove the ``if" direction.

Now, let $f, g, h \in \Inj(\Omega)$ be any three maps, each having at least one infinite cycle, such that $\, |\Omega \setminus (\Omega)f| + |\Omega \setminus (\Omega)g| = |\Omega \setminus (\Omega)h|$. By Lemma~\ref{prod-conj-equiv}, to show that there exist permutations $a, b \in \Sym(\Omega)$ such that $h = afa^{-1}bgb^{-1}$, it suffices to construct elements $f', g' \in  \Inj(\Omega)$ satisfying $f' \sim f$, $g' \sim g$, and $f'g' \sim h$. Our strategy in proving the ``if" direction of the theorem, in the cases not covered by Theorem~\ref{droste theorem}, will therefore be as follows. For every three relevant $\sim$-equivalence classes $A, B, C$, such that for all $f \in A$, $g \in B$, $h \in C$ one has $\, |\Omega \setminus (\Omega)f| + |\Omega \setminus (\Omega)g| = |\Omega \setminus (\Omega)h|$, we shall construct elements $f'$ and $g'$, such that $f' \in A$, $g' \in B$, and $f'g' \in C$. To do this, we shall first prove a series of lemmas in Sections~\ref{fin section}--\ref{fwd section} which reduce the problem to the case where $f'$ and $g'$ each have exactly one infinite cycle, and where none of $f'$, $g'$, and $h' = f'g'$ has any finite cycles. Then, in Section~\ref{construction section} we shall construct the desired maps $f'$ and $g'$ in that case. More specifically, the lemmas of Section~\ref{fin section} will allow us to add any number of finite cycles to the $f'$, $g'$, and $h'$ constructed in Section~\ref{construction section}. Similarly, those of Section~\ref{open section} will allow us to add any number of open cycles to $f'$ and $g'$.  Likewise, the lemmas of Section~\ref{fwd section} will allow us to add any number of forward cycles to $f'$ and $g'$ (and hence also to $h'$). Finally, in Section~\ref{proof section} we shall combine all these pieces, for a complete proof of the theorem.

\section{Finite cycles} \label{fin section}

From now on, $\Omega$ will denote a countably infinite set. The next lemma says, informally, that given maps $f,g,h \in \Inj(\Omega)$ such that $fg = h$, we can, in certain situations, add finite cycles to $h$ without changing the $\sim$-equivalence classes of $f$ and $g$.

\begin{lemma} \label{fin h}
Let $f, g, h \in \Inj(\Omega)$ be any three maps, each having at least one infinite cycle, such that $fg = h$, and for each $n \in \Z_+$, let $K_n$ be a cardinal satisfying $\, 0 \leq K_n \leq \aleph_0$. Further, suppose that there are infinite cycles $\, \Gamma_f$, $\, \Gamma_g$, and $\, \Gamma_h$ of $f$, $g$, and $h$, respectively, such that $\, \Sigma =  \Gamma_f \cap  \Gamma_g \cap  \Gamma_h$ is infinite. Then there exist maps $f', g', h' \in \Inj(\Omega)$ such that $f'g' = h'$, $f' \sim f$, $g' \sim g$, $(h')\C_{\mathrm{open}} = (h)\C_{\mathrm{open}} $, $(h')\C_{\mathrm{fwd}} = (h)\C_{\mathrm{fwd}}$, and $(h')\C_n = (h)\C_n + K_n$ for all $n \in \Z_+.$

Furthermore, the maps $f'$, $g'$, and $h'$ can be chosen so that there are infinite cycles $\, \Gamma_{f'}$, $\, \Gamma_{g'}$, and $\, \Gamma_{h'}$ of $f'$, $g'$, and $h'$, respectively, such that $\, \Gamma_{f'} \cap  \Gamma_{g'} \cap  \Gamma_{h'}$ is infinite.
\end{lemma}

\begin{proof}
We begin by proving the lemma in the case where there exists $n \in \Z_+$ such that $K_n = 1$, and for all $m \in \Z_+ \setminus \{n\}$ we have $K_m = 0$.

Let $\alpha \in \Sigma$ be any element that is not initial in any of $\Gamma_f$, $\Gamma_g$, and $\Gamma_h$ (if any of them is a forward cycle). Write $\beta = (\alpha)f^{-1}$, $\gamma = (\alpha)f$, $\delta = (\alpha)g^{-1}$, $\epsilon = (\alpha)g$, $\zeta  = (\alpha)h^{-1}$, and $\eta = (\alpha)h$. (Therefore, in particular, $\beta, \gamma \in \Gamma_f$, $\delta, \epsilon \in \Gamma_g$, and $\zeta, \eta \in \Gamma_h$.) Also, let $\Omega' = (\Omega \setminus \{\alpha\}) \dot{\cup} \{\theta_1, \theta_2, \dots, \theta_{2n}\}$. We shall next define injective endomaps $\bar{f}$, $\bar{g}$, and $\bar{h}$ on $\Omega'$.

Define $\bar{f}$ to act by
$$\left\{ \begin{array}{ll}
\beta \mapsto \theta_{n+1} & \\
\theta_i \mapsto \theta_{i+n+1} & \textrm{for } 1 \leq i \leq n-1\\
\theta_n \mapsto \gamma & \\
\theta_i \mapsto \theta_{i-n} & \textrm{for } n+1 \leq i \leq 2n\\
\lambda \mapsto (\lambda)f & \textrm{otherwise}
\end{array}\right.,$$ and define $\bar{g}$ to act by 
$$\left\{ \begin{array}{ll}
\delta \mapsto \theta_1 & \\
\theta_1 \mapsto \theta_{2n} & \\
\theta_i \mapsto \theta_{i+n-1} & \textrm{for } 2 \leq i \leq n\\
\theta_{n+1} \mapsto \epsilon & \\
\theta_i \mapsto \theta_{i-n} & \textrm{for } n+2 \leq i \leq 2n\\
\lambda \mapsto (\lambda)g & \textrm{otherwise}
\end{array}\right..$$

The action of $\bar{f}$ on $\{\theta_1, \theta_2, \dots, \theta_{2n}\}$ can be visualized as follows: $$\cdots \rightarrow \beta \rightarrow \fbox{$\theta_{n+1}$} \rightarrow \theta_1 \rightarrow \fbox{$\theta_{n+2}$} \rightarrow \theta_2 \rightarrow \cdots \rightarrow \fbox{$\theta_{2n-1}$} \rightarrow \theta_{n-1} \rightarrow \fbox{$\theta_{2n}$} \rightarrow \theta_n \rightarrow \gamma \rightarrow \cdots $$ (the boxes have been added to emphasize the pattern). (We should note that in the case where $n = 1$, the condition $1 \leq i \leq n-1$ in the definition of $\bar{f}$ is vacuous, and hence the above diagram, while still technically accurate as is, simplifies to: $$\cdots \rightarrow \beta \rightarrow \fbox{$\theta_{2}$} \rightarrow \theta_1\rightarrow \gamma \rightarrow \cdots .$$ Similar reductions occur with $\bar{g}$ and also with $\bar{h}$ below.) Thus, $f'$ is indeed a well-defined injective map on $\Omega'$. Also, $(\Gamma_f \setminus \{\alpha\}) \cup \{\theta_1, \theta_2, \dots, \theta_{2n}\}$ is an infinite cycle under $\bar{f}$ (of the same type as $\Gamma_f$ under $f$), and otherwise $\bar{f}$ has the same cycles as $f$. It follows that for any bijection $a : \Omega \rightarrow \Omega'$, we have $f \sim a\bar{f}a^{-1}$, by Lemma~\ref{pass}. Similarly, the action of $\bar{g}$ on $\{\theta_1, \theta_2, \dots, \theta_{2n}\}$ can be visualized as follows: $$\cdots \rightarrow \delta \rightarrow \theta_1 \rightarrow \theta_{2n} \rightarrow \fbox{$\theta_n$} \rightarrow \theta_{2n-1} \rightarrow \fbox{$\theta_{n-1}$} \rightarrow \cdots$$ $$\rightarrow \fbox{$\theta_3$} \rightarrow \theta_{n+2} \rightarrow \fbox{$\theta_2$} \rightarrow \theta_{n+1} \rightarrow \epsilon \rightarrow \cdots .$$ Thus, $g'$ is a well-defined injective map on $\Omega'$. Also, $(\Gamma_g \setminus \{\alpha\}) \cup \{\theta_1, \theta_2, \dots, \theta_{2n}\}$ is an infinite cycle under $\bar{g}$ (of the same type as $\Gamma_g$ under $g$), and otherwise $\bar{g}$ has the same cycles as $g$.  It follows that for any bijection $a : \Omega \rightarrow \Omega'$, we have $g \sim a\bar{g}a^{-1}$, by Lemma~\ref{pass}.

Now, set $\bar{h} = \bar{f}\bar{g}$. Noting that $\zeta = (\delta)f^{-1}$, $\eta = (\gamma)g$, and $(\beta)\bar{h} = (\beta)\bar{f}\bar{g} = (\theta_{n+1})\bar{g} = \epsilon = (\beta)h$, we see that $\bar{h}$ acts by 
$$\left\{ \begin{array}{ll}
\zeta \mapsto \theta_1 & \\
\theta_i \mapsto \theta_{i+1} & \textrm{for } 1 \leq i \leq n-1\\
\theta_n \mapsto \eta & \\
\theta_{n+1} \mapsto \theta_{2n} & \\
\theta_i \mapsto \theta_{i-1} & \textrm{for } n+2 \leq i \leq 2n\\
\lambda \mapsto (\lambda)h & \textrm{otherwise}
\end{array}\right..$$ On $\{\theta_1, \theta_2, \dots, \theta_n\}$ the action of $\bar{h}$ can be visualized as follows: $$\cdots \rightarrow \zeta \rightarrow\theta_1 \rightarrow \theta_2 \rightarrow \cdots \rightarrow \theta_{n-1} \rightarrow \theta_n \rightarrow \eta \rightarrow \cdots .$$ Thus, $(\Gamma_h \setminus \{\alpha\}) \cup \{\theta_1, \theta_2, \dots, \theta_n\}$ is an infinite cycle under $\bar{h}$, $\{\theta_{n+1}, \theta_{n+2}, \dots, \theta_{2n}\}$ is a cycle of cardinality $n$ under $\bar{h}$, and otherwise $\bar{h}$ has the same cycles as $h$.  It follows that for any bijection $a : \Omega \rightarrow \Omega'$, we have $(a\bar{h}a^{-1})\C_{\mathrm{open}} = (h)\C_{\mathrm{open}} $, $(a\bar{h}a^{-1})\C_{\mathrm{fwd}} = (h)\C_{\mathrm{fwd}}$, $(a\bar{h}a^{-1})\C_n = (h)\C_n + 1$, and $(a\bar{h}a^{-1})\C_m = (h)\C_m $ for all $m \in \Z_+ \setminus \{n\}.$ Also $a\bar{f}a^{-1}a\bar{g}a^{-1} = a\bar{h}a^{-1}$. Thus, setting $f' = a\bar{f}a^{-1}$, $g' = a\bar{g}a^{-1}$, and $h' = a\bar{h}a^{-1}$, we have proved the lemma in the desired special case.

For the general case, we may assume that not all of the $K_n$ are zero, since otherwise there is nothing to prove. Write $\Sigma = (\bigcup_{n=1}^\infty \Sigma_n) \cup \Delta$, where the union is disjoint, for each $n \in \Z_+$, $|\Sigma_n| = K_n$, and $|\Delta| = \aleph_0$. Moreover, since $\Sigma$ is infinite, we can choose $\bigcup_{n=1}^\infty \Sigma_n$ in such a way that each $\alpha \in \bigcup_{n=1}^\infty \Sigma_n$ is not initial in $\Gamma_f$, $\Gamma_g$, and $\Gamma_h$ (if any of them is a forward cycle), and so that for all $\alpha, \alpha' \in \bigcup_{n=1}^\infty \Sigma_n$ we have $\alpha' \notin \{(\alpha)f, (\alpha)g, (\alpha)h, (\alpha)f^{-1}, (\alpha)g^{-1}, (\alpha)h^{-1}\}$. For each $n \in \Z_+$, write $\Sigma_n = \{\alpha^{n,1}, \alpha^{n,2}, \alpha^{n,3}, \dots \}$. Then, we can perform the above construction for each $\alpha^{n,j}$ simultaneously (by simply adding the superscript ``$n,j$" to each $\beta$, $\gamma$, $\delta$, $\epsilon$, $\zeta$, $\eta$, and $\theta_i$ above) and define the maps $\bar{f}$, $\bar{g}$, $\bar{h}$ on $$\Omega' = (\Omega \setminus \bigcup_{n=1}^\infty \Sigma_n) \cup \bigcup_{n,j} \Delta_{n,j},$$ where $\Delta_{n,j} = \{\theta_1^{n,j}, \theta_2^{n,j}, \dots, \theta_{2n}^{n,j}\}$ is the $2n$-element set corresponding to $\alpha^{n,j}$. The resulting $\bar{f}$ will have $$\Gamma_{\bar{f}} = (\Gamma_f \setminus \bigcup_{n=1}^\infty \Sigma_n) \cup \bigcup_{n,j} \Delta_{n,j}$$ as an infinite cycle, and otherwise $\bar{f}$ will have the same cycles as $f$. Also, $\bar{g}$ will have $$\Gamma_{\bar{g}} = (\Gamma_g \setminus \bigcup_{n=1}^\infty \Sigma_n) \cup \bigcup_{n,j} \Delta_{n,j}$$ as an infinite cycle, and otherwise $\bar{g}$ will have the same cycles as $g$. Similarly, $\bar{h} = \bar{f}\bar{g}$ will have $$\Gamma_{\bar{h}} = (\Gamma_h \setminus \bigcup_{n=1}^\infty \Sigma_n) \cup \bigcup_{n,j} \{\theta_1^{n,j}, \theta_2^{n,j}, \dots, \theta_{n}^{n,j}\}$$ as an infinite cycle and each $\{\theta_{n+1}^{n,j}, \theta_{n+2}^{n,j}, \dots, \theta_{2n}^{n,j}\}$ as a cycle of cardinality $n$, and otherwise $\bar{h}$ will have the same cycles as $h$. We also note that $\Delta \subseteq \Gamma_{\bar{f}} \cap \Gamma_{\bar{g}} \cap \Gamma_{\bar{h}}$, and hence the latter is infinite. Thus, for any bijection $a : \Omega \rightarrow \Omega'$, the maps $f' = a\bar{f}a^{-1}$, $g' = a\bar{g}a^{-1}$, and $h' = a\bar{h}a^{-1}$ will have the desired properties.
\end{proof}

We could have actually proved the above lemma with slightly weaker hypotheses. Namely, instead of assuming that there are cycles of $f$, $g$, and $h$ whose intersection is infinite, we could have assumed that there are infinitely many elements $\alpha \in \Omega$, each of which is contained in some infinite cycle under each of $f$, $g$, and $h$, and that these elements $\alpha$ are not initial in those cycles. Proving this more general version of the lemma would require essentially no additional work, however the formulation we have given suffices for our present purposes. A similar remark can be made about each subsequent lemma. Also, the proofs of all the lemmas that follow will be similar to the one above -- the only significant differences will be in how $\bar{f}$ and $\bar{g}$ are defined. Thus, fewer details will be given in those proofs. 

The next lemma is an analogue of Lemma~\ref{fin h} that allows us to add finite cycles to $f$, rather than $h$, in the equation $fg = h$.

\begin{lemma} \label{fin f}
Let $f, g, h \in \Inj(\Omega)$ be any three maps, each having at least one infinite cycle, such that $fg = h$, and for each $n \in \Z_+$, let $K_n$ be a cardinal satisfying $\, 0 \leq K_n \leq \aleph_0$. Further, suppose that there are infinite cycles $\, \Gamma_f$, $\, \Gamma_g$, and $\, \Gamma_h$ of $f$, $g$, and $h$, respectively, such that $\, \Sigma =  \Gamma_f \cap  \Gamma_g \cap  \Gamma_h$ is infinite. Then there exist maps  $f', g', h' \in \Inj(\Omega)$ such that $f'g' = h'$, $g' \sim g$, $h' \sim h$, $(f')\C_{\mathrm{open}} = (f)\C_{\mathrm{open}} $, $(f')\C_{\mathrm{fwd}} = (f)\C_{\mathrm{fwd}}$, and $(f')\C_n = (f)\C_n + K_n$ for all $n \in \Z_+.$

Furthermore, the maps $f'$, $g'$, and $h'$ can be chosen so that there are infinite cycles $\, \Gamma_{f'}$, $\, \Gamma_{g'}$, and $\, \Gamma_{h'}$ of $f'$, $g'$, and $h'$, respectively, such that $\, \Gamma_{f'} \cap  \Gamma_{g'} \cap  \Gamma_{h'}$ is infinite.
\end{lemma}

\begin{proof}
As with Lemma~\ref{fin h}, we begin by proving the lemma in the case where there exists $n \in \Z_+$ such that $K_n = 1$, and for all $m \in \Z_+ \setminus \{n\}$ we have $K_m = 0$. 

Again, let $\alpha \in \Sigma$ be any element that is not initial in any of $\Gamma_f$, $\Gamma_g$, and $\Gamma_h$ (if any of them is a forward cycle). Write $\beta = (\alpha)f^{-1}$, $\gamma = (\alpha)f$, $\delta = (\alpha)g^{-1}$, $\epsilon = (\alpha)g$, $\zeta  = (\alpha)h^{-1}$, and $\eta = (\alpha)h$. Also, let $\Omega' = (\Omega \setminus \{\alpha\}) \dot{\cup} \{\theta_1, \theta_2, \dots, \theta_{2n}\}$. We shall next define injective endomaps $\bar{f}$, $\bar{g}$, and $\bar{h}$ on $\Omega'$.

Define $\bar{f}$ to act by $$\left\{ \begin{array}{ll}
\beta \mapsto \theta_{1} & \\
\theta_i \mapsto \theta_{i+1} & \textrm{for } 1 \leq i \leq n-1, n+1 \leq i \leq 2n-1 \\
\theta_n \mapsto \gamma & \\
\theta_{2n} \mapsto \theta_{n+1} & \\
\lambda \mapsto (\lambda)f & \textrm{otherwise}
\end{array}\right.,$$ and define $\bar{g}$ to act by 
$$\left\{ \begin{array}{ll}
\delta \mapsto \theta_{n+1} & \\
\theta_1 \mapsto \epsilon & \\
\theta_i \mapsto \theta_{2n+2-i} & \textrm{for } 2 \leq i \leq n\\
\theta_i \mapsto \theta_{2n+1-i} & \textrm{for } n+1 \leq i \leq 2n\\
\lambda \mapsto (\lambda)g & \textrm{otherwise}
\end{array}\right..$$

The action of $\bar{f}$ on $\{\theta_1, \theta_2, \dots, \theta_{n}\}$ can be visualized as follows: $$\cdots \rightarrow \beta \rightarrow \theta_1 \rightarrow \theta_2 \rightarrow \cdots \rightarrow \theta_{n-1} \rightarrow \theta_n \rightarrow \gamma \rightarrow \cdots .$$ Thus, $(\Gamma_f \setminus \{\alpha\}) \cup \{\theta_1, \theta_2, \dots, \theta_{n}\}$ is an infinite cycle under $\bar{f}$, $\{\theta_{n+1}, \theta_{n+2}, \dots, \theta_{2n}\}$ is a cycle of cardinality $n$ under $\bar{f}$, and otherwise $\bar{f}$ has the same cycles as $f$.  It follows from Lemma~\ref{pass} that for any bijection $a : \Omega \rightarrow \Omega'$, we have $(a\bar{f}a^{-1})\C_{\mathrm{open}} = (f)\C_{\mathrm{open}} $, $(a\bar{f}a^{-1})\C_{\mathrm{fwd}} = (f)\C_{\mathrm{fwd}}$, $(a\bar{f}a^{-1})\C_n = (f)\C_n + 1$, and $(a\bar{f}a^{-1})\C_m = (f)\C_m $ for all $m \in \Z_+ \setminus \{n\}.$ Similarly, the action of $\bar{g}$ on $\{\theta_1, \theta_2, \dots, \theta_{2n}\}$ can be visualized as follows: $$\cdots \rightarrow \delta \rightarrow \fbox{$\theta_{n+1}$} \rightarrow \theta_{n} \rightarrow \fbox{$\theta_{n+2}$} \rightarrow \theta_{n-1} \rightarrow \cdots \rightarrow \fbox{$\theta_{2n-1}$} \rightarrow \theta_{2} \rightarrow \fbox{$\theta_{2n}$} \rightarrow \theta_{1} \rightarrow \epsilon \rightarrow \cdots .$$ Thus, $(\Gamma_g \setminus \{\alpha\}) \cup \{\theta_1, \theta_2, \dots, \theta_{2n}\}$ is an infinite cycle under $\bar{g}$, and otherwise $\bar{g}$ has the same cycles as $g$.  It follows that for any bijection $a : \Omega \rightarrow \Omega'$, we have $g \sim a\bar{g}a^{-1}$.

Now, set $\bar{h} = \bar{f}\bar{g}$. Noting that $\zeta = (\delta)f^{-1}$, $\eta = (\gamma)g$, and $(\beta)\bar{h} = (\beta)\bar{f}\bar{g} = (\theta_{1})\bar{g} = \epsilon = (\beta)h$, we see that $\bar{h}$ acts by 
$$\left\{ \begin{array}{ll}
\zeta \mapsto \theta_{n+1} & \\
\theta_i \mapsto \theta_{2n+1-i} & \textrm{for } 1 \leq i \leq n-1\\
\theta_n \mapsto \eta & \\
\theta_i \mapsto \theta_{2n-i} & \textrm{for } n+1 \leq i \leq 2n-1\\
\theta_{2n} \mapsto \theta_{n} & \\
\lambda \mapsto (\lambda)h & \textrm{otherwise}
\end{array}\right..$$ On $\{\theta_1, \theta_2, \dots, \theta_{2n}\}$ the action of $\bar{h}$ can be visualized as follows:$$\cdots \rightarrow \zeta \rightarrow \fbox{$\theta_{n+1}$} \rightarrow \theta_{n-1} \rightarrow \fbox{$\theta_{n+2}$} \rightarrow \theta_{n-2} \rightarrow \cdots \rightarrow \fbox{$\theta_{2n-1}$} \rightarrow \theta_{1} \rightarrow \fbox{$\theta_{2n}$} \rightarrow \theta_{n} \rightarrow \eta \rightarrow \cdots .$$ Thus, $(\Gamma_h \setminus \{\alpha\}) \cup \{\theta_1, \theta_2, \dots, \theta_{2n}\}$ is an infinite cycle under $\bar{h}$, and otherwise $\bar{h}$ has the same cycles as $h$.  It follows that for any bijection $a : \Omega \rightarrow \Omega'$, we have $h \sim a\bar{h}a^{-1}$, by Lemma~\ref{pass}. Also, $a\bar{f}a^{-1}a\bar{g}a^{-1} = a\bar{h}a^{-1}$. Thus, setting $f' = a\bar{f}a^{-1}$, $g' = a\bar{g}a^{-1}$, and $h' = a\bar{h}a^{-1}$, we have proved the lemma in the desired special case.

The general case and the final claim can be obtained by the same argument as in the last paragraph of the proof of Lemma~\ref{fin h}, so we omit the details here, to minimize repetition.
\end{proof}

The next lemma is an analogue of Lemmas~\ref{fin h} and~\ref{fin f} that allows us to add finite cycles to $g$ in the equation $fg = h$.

\begin{lemma} \label{fin g}
Let $f, g, h \in \Inj(\Omega)$ be any three maps, each having at least one infinite cycle, such that $fg = h$, and for each $n \in \Z_+$, let $K_n$ be a cardinal satisfying $\, 0 \leq K_n \leq \aleph_0$. Further, suppose that there are infinite cycles $\, \Gamma_f$, $\, \Gamma_g$, and $\, \Gamma_h$ of $f$, $g$, and $h$, respectively, such that $\, \Sigma =  \Gamma_f \cap  \Gamma_g \cap  \Gamma_h$ is infinite. Then there exist maps  $f', g', h' \in \Inj(\Omega)$ such that $f'g' = h'$, $f' \sim f$, $h' \sim h$, $(g')\C_{\mathrm{open}} = (g)\C_{\mathrm{open}} $, $(g')\C_{\mathrm{fwd}} = (g)\C_{\mathrm{fwd}}$, and $(g')\C_n = (g)\C_n + K_n$ for all $n \in \Z_+.$

Furthermore, the maps $f'$, $g'$, and $h'$ can be chosen so that there are infinite cycles $\, \Gamma_{f'}$, $\, \Gamma_{g'}$, and $\, \Gamma_{h'}$ of $f'$, $g'$, and $h'$, respectively, such that $\, \Gamma_{f'} \cap  \Gamma_{g'} \cap  \Gamma_{h'}$ is infinite.
\end{lemma}

\begin{proof}
As before, we shall prove the lemma in the case where there exists $n \in \Z_+$ such that $K_n = 1$, and for all $m \in \Z_+ \setminus \{n\}$ we have $K_m = 0$. 

Again, let $\alpha \in \Sigma$ be any element that is not initial in any of $\Gamma_f$, $\Gamma_g$, and $\Gamma_h$ (if any of them is a forward cycle). Write $\beta = (\alpha)f^{-1}$, $\gamma = (\alpha)f$, $\delta = (\alpha)g^{-1}$, $\epsilon = (\alpha)g$, $\zeta  = (\alpha)h^{-1}$, and $\eta = (\alpha)h$. Also, let $\Omega' = (\Omega \setminus \{\alpha\}) \dot{\cup} \{\theta_1, \theta_2, \dots, \theta_{2n}\}$. We shall next define injective endomaps $\bar{f}$, $\bar{g}$, and $\bar{h}$ on $\Omega'$.

Define $\bar{f}$ to act by 
$$\left\{ \begin{array}{ll}
\beta \mapsto \theta_{n} & \\
\theta_i \mapsto \theta_{2n+1-i} & \textrm{for } 1 \leq i \leq n\\
\theta_i \mapsto \theta_{2n-i} & \textrm{for } n+1 \leq i \leq 2n-1\\
\theta_{2n} \mapsto \gamma & \\
\lambda \mapsto(\lambda)f & \textrm{otherwise}
\end{array}\right.,$$ and define $\bar{g}$ to act by 
$$\left\{ \begin{array}{ll}
\delta \mapsto \theta_{1} & \\
\theta_i \mapsto \theta_{i+1} &\textrm{for } 1 \leq i \leq n-1, n+1 \leq i \leq 2n-1\\
\theta_n \mapsto \epsilon & \\
\theta_{2n} \mapsto \theta_{n+1} & \\
\lambda \mapsto (\lambda)g & \textrm{otherwise}
\end{array}\right..$$

The action of $\bar{f}$ on $\{\theta_1, \theta_2, \dots, \theta_{2n}\}$ can be visualized as follows: $$\cdots \rightarrow \beta \rightarrow \fbox{$\theta_{n}$} \rightarrow \theta_{n+1} \rightarrow \fbox{$\theta_{n-1}$} \rightarrow \theta_{n+2} \rightarrow \cdots \rightarrow \fbox{$\theta_{2}$} \rightarrow \theta_{2n-1} \rightarrow \fbox{$\theta_{1}$} \rightarrow \theta_{2n} \rightarrow \gamma \rightarrow \cdots .$$ Thus, $(\Gamma_f \setminus \{\alpha\}) \cup \{\theta_1, \theta_2, \dots, \theta_{2n}\}$ is an infinite cycle under $\bar{f}$, and otherwise $\bar{f}$ has the same cycles as $f$. It follows that for any bijection $a : \Omega \rightarrow \Omega'$, we have $f \sim a\bar{f}a^{-1}$. Similarly, the action of $\bar{g}$ on $\{\theta_1, \theta_2, \dots, \theta_{n}\}$ can be visualized as follows: $$\cdots \rightarrow \delta \rightarrow \theta_{1} \rightarrow \theta_{2} \rightarrow \cdots \rightarrow \theta_{n} \rightarrow \epsilon \rightarrow \cdots .$$ Thus, $(\Gamma_g \setminus \{\alpha\}) \cup \{\theta_1, \theta_2, \dots, \theta_{n}\}$ is an infinite cycle under $\bar{g}$, $\{\theta_{n+1}, \theta_{n+2}, \dots, \theta_{2n}\}$ is a cycle of cardinality $n$ under $\bar{g}$, and otherwise $\bar{g}$ has the same cycles as $g$. It follows that for any bijection $a : \Omega \rightarrow \Omega'$, we have $(a\bar{g}a^{-1})\C_{\mathrm{open}} = (g)\C_{\mathrm{open}} $, $(a\bar{g}a^{-1})\C_{\mathrm{fwd}} = (g)\C_{\mathrm{fwd}}$, $(a\bar{g}a^{-1})\C_n = (g)\C_n + 1$, and $(a\bar{g}a^{-1})\C_m = (g)\C_m $ for all $m \in \Z_+ \setminus \{n\}.$

Setting $\bar{h} = \bar{f}\bar{g}$, we see that $\bar{h}$ acts by 
$$\left\{ \begin{array}{ll}
\zeta \mapsto \theta_{1} & \\
\theta_1 \mapsto \theta_{n+1} & \\
\theta_i \mapsto \theta_{2n+2-i} & \textrm{for } 2 \leq i \leq n\\
\theta_i \mapsto \theta_{2n+1-i} & \textrm{for } n+1 \leq i \leq 2n-1\\
\theta_{2n} \mapsto \eta & \\
\lambda \mapsto (\lambda)h & \textrm{otherwise}
\end{array}\right..$$
On $\{\theta_1, \theta_2, \dots, \theta_{2n}\}$ the action of $\bar{h}$ can be visualized as follows:$$\cdots \rightarrow \zeta \rightarrow \theta_1 \rightarrow \fbox{$\theta_{n+1}$} \rightarrow \theta_{n} \rightarrow \fbox{$\theta_{n+2}$} \rightarrow \theta_{n-1} \rightarrow \cdots \rightarrow \fbox{$\theta_{2n-1}$} \rightarrow \theta_{2} \rightarrow \fbox{$\theta_{2n}$} \rightarrow \eta \rightarrow \cdots .$$ Thus, $(\Gamma_h \setminus \{\alpha\}) \cup \{\theta_1, \theta_2, \dots, \theta_{2n}\}$ is an infinite cycle under $\bar{h}$, and otherwise $\bar{h}$ has the same cycles as $h$.  It follows that for any bijection $a : \Omega \rightarrow \Omega'$, we have $h \sim a\bar{h}a^{-1}$. Also $a\bar{f}a^{-1}a\bar{g}a^{-1} = a\bar{h}a^{-1}$. Thus, setting $f' = a\bar{f}a^{-1}$, $g' = a\bar{g}a^{-1}$, and $h' = a\bar{h}a^{-1}$, we have proved the lemma in the desired special case.

As before, the general case and the final claim can be obtained by the same argument as in the last paragraph of the proof of Lemma~\ref{fin h}.
\end{proof}

\section{Open cycles} \label{open section}

The lemmas in this section are analogous to those of the previous one, except now we shall be adding open, rather than finite, cycles to our maps $f$ and $g$ in the equation $fg = h$. (Consideration of the open cycles of $h$ will be postponed until Section~\ref{construction section}.) We begin with $f$.

\begin{lemma} \label{open f}
Let $f, g, h \in \Inj(\Omega)$ be any three maps, each having at least one infinite cycle, such that $fg = h$, and let $K$ be a cardinal satisfying $\, 1 \leq K \leq \aleph_0$. Further, suppose that there are infinite cycles $\, \Gamma_f$, $\, \Gamma_g$, and $\, \Gamma_h$ of $f$, $g$, and $h$, respectively, such that $\, \Sigma =  \Gamma_f \cap  \Gamma_g \cap  \Gamma_h$ is infinite. Then there exist maps  $f', g', h' \in \Inj(\Omega)$ such that $f'g' = h'$, $g' \sim g$, $h' \sim h$, $(f')\C_{\mathrm{open}} = (f)\C_{\mathrm{open}}  + K$, $(f')\C_{\mathrm{fwd}} = (f)\C_{\mathrm{fwd}}$, and $(f')\C_n = (f)\C_n$ for all $n \in \Z_+.$

Furthermore, the maps $f'$, $g'$, and $h'$ can be chosen so that there are infinite cycles $\, \Gamma_{f'}$, $\, \Gamma_{g'}$, and $\, \Gamma_{h'}$ of $f'$, $g'$, and $h'$, respectively, such that $\, \Gamma_{f'} \cap  \Gamma_{g'} \cap  \Gamma_{h'}$ is infinite.
\end{lemma}

\begin{proof}
Let us begin by proving the lemma in the case where $K = 1$. 

Write $\Sigma = \{\alpha_i\}_{i\in \Z_+} \cup \Delta$, where the union is disjoint and $|\Delta| = \aleph_0$. We may choose the elements $\alpha_i$ so that none of them is initial in $\Gamma_g$ and $\Gamma_h$ (if either of them is a forward cycle). Similarly, we may assume that for all $\alpha_i$ and $\alpha_j$ we have $\alpha_j \notin \{(\alpha_i)g^{-1}, (\alpha_i)h^{-1}\}$. For each $i \in \Z_+$, write $\delta_i = (\alpha_i)g^{-1}$ and $\zeta_i  = (\alpha_i)h^{-1}$. Also, let $\Omega' = \Omega \dot{\cup} \{\theta_i\}_{i\in \Z}$. We shall next define injective endomaps $\bar{f}$, $\bar{g}$, and $\bar{h}$ on $\Omega'$.

Define $\bar{f}$ to act by 
$$\left\{ \begin{array}{ll}
\theta_i \mapsto \theta_{i+1} & \textrm{for } i\in \Z\\
\lambda \mapsto (\lambda)f & \textrm{for } \lambda \in \Omega
\end{array}\right.,$$ and define $\bar{g}$ to act by 
$$\left\{ \begin{array}{ll}
\delta_i \mapsto \theta_{-i+1} & \textrm{for } i \geq 1\\
\theta_i \mapsto \alpha_{i} & \textrm{for } i \geq 1\\
\theta_{i} \mapsto \theta_{-i+1} & \textrm{for } i \leq 0\\
\lambda \mapsto (\lambda)g & \textrm{otherwise}
\end{array}\right..$$

Thus, $\{\theta_i\}_{i\in \Z}$ is an open cycle under $\bar{f}$, and otherwise $\bar{f}$ has the same cycles as $f$. It follows that for any bijection $a : \Omega \rightarrow \Omega'$, we have $(a\bar{f}a^{-1})\C_{\mathrm{open}} = (f)\C_{\mathrm{open}}  + 1$, $(a\bar{f}a^{-1})\C_{\mathrm{fwd}} = (f)\C_{\mathrm{fwd}}$, and $(a\bar{f}a^{-1})\C_n = (f)\C_n$ for all $n \in \Z_+$, by Lemma~\ref{pass}. For each $i \in \Z_+$, the action of $\bar{g}$ on the elements $\alpha_i$, $\delta_i$, and $\theta_i$ can be visualized as follows: $$\cdots \rightarrow \delta_i \rightarrow \theta_{-i+1} \rightarrow \theta_{i} \rightarrow \alpha_i \rightarrow \cdots .$$ Thus, $\Gamma_g \cup \{\theta_i\}_{i\in \Z}$ is an infinite cycle under $\bar{g}$, and otherwise $\bar{g}$ has the same cycles as $g$. It follows that for any bijection $a : \Omega \rightarrow \Omega'$, we have $g \sim a\bar{g}a^{-1}$, by Lemma~\ref{pass}.

Setting $\bar{h} = \bar{f}\bar{g}$, we see that $\bar{h}$ acts by 
$$\left\{ \begin{array}{ll}
\zeta_i \mapsto \theta_{-i+1} & \textrm{for } i \geq 1 \\
\theta_i \mapsto \alpha_{i+1} & \textrm{for } i \geq 0\\
\theta_i \mapsto \theta_{-i} & \textrm{for } i \leq -1\\
\lambda \mapsto (\lambda)h & \textrm{otherwise}
\end{array}\right..$$
For each $i \geq 2$, the action of $\bar{h}$ on the elements $\alpha_i$, $\zeta_i$, and $\theta_i$ can be visualized as follows: $$\cdots \rightarrow \zeta_i \rightarrow \theta_{-i+1} \rightarrow \theta_{i-1} \rightarrow \alpha_i \rightarrow \cdots .$$ Also, for $i = 1$ we have $$\cdots \rightarrow \zeta_1 \rightarrow \theta_0 \rightarrow \alpha_1 \rightarrow \cdots .$$ Thus, $\Gamma_h \cup \{\theta_i\}_{i\in \Z}$ is an infinite cycle under $\bar{h}$, and otherwise $\bar{h}$ has the same cycles as $h$. It follows that for any bijection $a : \Omega \rightarrow \Omega'$, we have $h \sim a\bar{h}a^{-1}$. Therefore, setting $f' = a\bar{f}a^{-1}$, $g' = a\bar{g}a^{-1}$, and $h' = a\bar{h}a^{-1}$, we have proved the lemma in the desired special case.

For the general case, write $\Sigma = (\bigcup_{n=1}^K \Sigma_n) \cup \Delta$, where the union is disjoint, each $|\Sigma_n| = \aleph_0$, and $|\Delta| = \aleph_0$. Moreover, since $\Sigma$ is infinite, we can choose $\bigcup_{n=1}^K \Sigma_n$ in such a way that no $\alpha\in \bigcup_{n=1}^K \Sigma_n$ is initial in $\Gamma_g$ or $\Gamma_h$ (if either of them is a forward cycle), and so that for all $\alpha, \alpha' \in \bigcup_{n=1}^K \Sigma_n$ we have $\alpha' \notin \{(\alpha)g^{-1}, (\alpha)h^{-1}\}$. For each $n \in \Z_+$, write $\Sigma_n = \{\alpha^n_i\}_{i\in \Z_+}$. Then, we can perform the above construction for each $\Sigma_n$ simultaneously (by simply adding the superscript ``$n$" to each $\delta_i$, $\zeta_i$, and $\theta_i$ above) and define the maps $\bar{f}$, $\bar{g}$, $\bar{h}$ on $\Omega' = \Omega \cup \bigcup_{n=1}^K \Delta_n$, where $\Delta_n = \{\theta_i^n\}_{i\in \Z}$ is the set corresponding to $\Sigma_n$. The resulting $\bar{f}$ will have each $\Delta_n$ as an open cycle, and otherwise $\bar{f}$ will have the same cycles as $f$. Therefore, $(\bar{f})\C_{\mathrm{open}} = (f)\C_{\mathrm{open}}  + K$, $(\bar{f})\C_{\mathrm{fwd}} = (f)\C_{\mathrm{fwd}}$, and $(\bar{f})\C_n = (f)\C_n$ for all $n \in \Z_+$. Also, $\bar{g}$ will have $\Gamma_{\bar{g}} = \Gamma_g \cup \bigcup_{n=1}^K \Delta_n$ as an infinite cycle, and otherwise $\bar{g}$ will have the same cycles as $g$. Similarly, $\bar{h}$ will have $\Gamma_{\bar{h}} = \Gamma_h \cup \bigcup_{n=1}^K \Delta_n$ as an infinite cycle, and otherwise $\bar{h}$ will have the same cycles as $h$. We also note that $\Delta \subseteq \Gamma_{f} \cap \Gamma_{g} \cap \Gamma_{h} \subseteq \Gamma_{\bar{f}} \cap \Gamma_{\bar{g}} \cap \Gamma_{\bar{h}}$, and hence the latter is infinite. Thus, for any bijection $a : \Omega \rightarrow \Omega'$, the maps $f' = a\bar{f}a^{-1}$, $g' = a\bar{g}a^{-1}$, and $h' = a\bar{h}a^{-1}$ will have the desired properties.
\end{proof}

We did not actually use the hypothesis that $\Gamma_f \cap  \Gamma_g \cap  \Gamma_h$ is infinite to prove the main claim of the above lemma -- only that $\Gamma_g \cap  \Gamma_h$ is infinite. The full strength of the hypothesis was only needed to prove the final claim of the lemma, the importance of which will become apparent in Section~\ref{construction section}. A similar observation applies to the next lemma, which is an analogue of Lemma~\ref{open f} for $g$, in the equation $fg = h$, as well as to the lemmas in Section~\ref{fwd section}.

\begin{lemma} \label{open g}
Let $f, g, h \in \Inj(\Omega)$ be any three maps, each having at least one infinite cycle, such that $fg = h$, and let $K$ be a cardinal satisfying $\, 1 \leq K \leq \aleph_0$. Further, suppose that there are infinite cycles $\, \Gamma_f$, $\, \Gamma_g$, and $\, \Gamma_h$ of $f$, $g$, and $h$, respectively, such that $\, \Sigma =  \Gamma_f \cap  \Gamma_g \cap  \Gamma_h$ is infinite. Then there exist maps  $f', g', h' \in \Inj(\Omega)$ such that $f'g' = h'$, $f' \sim f$, $h' \sim h$, $(g')\C_{\mathrm{open}} = (g)\C_{\mathrm{open}}  + K$, $(g')\C_{\mathrm{fwd}} = (g)\C_{\mathrm{fwd}}$, and $(g')\C_n = (g)\C_n$ for all $n \in \Z_+.$

Furthermore, the maps $f'$, $g'$, and $h'$ can be chosen so that there are infinite cycles $\, \Gamma_{f'}$, $\, \Gamma_{g'}$, and $\, \Gamma_{h'}$ of $f'$, $g'$, and $h'$, respectively, such that $\, \Gamma_{f'} \cap  \Gamma_{g'} \cap  \Gamma_{h'}$ is infinite.
\end{lemma}

\begin{proof}
As before, let us begin by proving the lemma in the case where $K = 1$. 

Write $\Sigma = \{\alpha_i\}_{i\in \Z_+} \cup \Delta$, where the union is disjoint and $|\Delta| = \aleph_0$. We may choose the elements $\alpha_i$ so that for all $\alpha_i$ and $\alpha_j$ we have $\alpha_j \notin \{(\alpha_i)f, (\alpha_i)h\}$. For each $i \in \Z_+$, write $\gamma_i = (\alpha_i)f$ and $\eta_i  = (\alpha_i)h$. Also, let $\Omega' = \Omega \dot{\cup} \{\theta_i\}_{i\in \Z}$. We shall next define injective endomaps $\bar{f}$, $\bar{g}$, and $\bar{h}$ on $\Omega'$.

Define $\bar{f}$ to act by 
$$\left\{ \begin{array}{ll}
\alpha_i \mapsto \theta_{-i+1} & \textrm{for } i \geq 1\\
\theta_i \mapsto \gamma_{i} & \textrm{for } i \geq 1\\
\theta_{i} \mapsto \theta_{-i+1} & \textrm{for } i \leq 0\\
\lambda \mapsto (\lambda)f & \textrm{otherwise}
\end{array}\right.,$$ and define $\bar{g}$ to act by 
$$\left\{ \begin{array}{ll}
\theta_i \mapsto \theta_{i+1} & \textrm{for } i\in \Z\\
\lambda \mapsto (\lambda)g & \textrm{for } \lambda \in \Omega
\end{array}\right..$$

For each $i \in \Z_+$, the action of $\bar{f}$ on the elements $\alpha_i$, $\gamma_i$, and $\theta_i$ can be visualized as follows: $$\cdots \rightarrow \alpha_i \rightarrow \theta_{-i+1} \rightarrow \theta_{i} \rightarrow \gamma_i \rightarrow \cdots .$$ Thus, $\Gamma_f \cup \{\theta_i\}_{i\in \Z}$ is an infinite cycle under $\bar{f}$, and otherwise $\bar{f}$ has the same cycles as $f$. It follows that for any bijection $a : \Omega \rightarrow \Omega'$, we have $f \sim a\bar{f}a^{-1}$. Also, clearly, $\{\theta_i\}_{i\in \Z}$ is an open cycle under $\bar{g}$, and otherwise $\bar{g}$ has the same cycles as $g$. Therefore, for any bijection $a : \Omega \rightarrow \Omega'$, we have $(a\bar{g}a^{-1})\C_{\mathrm{open}} = (g)\C_{\mathrm{open}}  + 1$, $(a\bar{g}a^{-1})\C_{\mathrm{fwd}} = (g)\C_{\mathrm{fwd}}$, and $(a\bar{g}a^{-1})\C_n = (g)\C_n$ for all $n \in \Z_+.$

Setting $\bar{h} = \bar{f}\bar{g}$, we see that $\bar{h}$ acts by 
$$\left\{ \begin{array}{ll}
\alpha_i \mapsto \theta_{-i+2} & \textrm{for } i \geq 1\\
\theta_i \mapsto \eta_{i} & \textrm{for } i \geq 1\\
\theta_i \mapsto \theta_{-i+2} & \textrm{for } i \leq 0\\
\lambda \mapsto (\lambda)h & \textrm{otherwise}
\end{array}\right..$$
For each $i \geq 2$, the action of $\bar{h}$ on the elements $\alpha_i$, $\eta_i$, and $\theta_i$ can be visualized as follows: $$\cdots \rightarrow \alpha_i \rightarrow \theta_{-i+2} \rightarrow \theta_{i} \rightarrow \eta_i \rightarrow \cdots .$$ Also, for $i = 1$ we have $$\cdots \rightarrow \alpha_1 \rightarrow \theta_1 \rightarrow \eta_1 \rightarrow \cdots .$$ Thus, $\Gamma_h \cup \{\theta_i\}_{i\in \Z}$ is an infinite cycle under $\bar{h}$, and otherwise $\bar{h}$ has the same cycles as $h$. It follows that for any bijection $a : \Omega \rightarrow \Omega'$, we have $h \sim a\bar{h}a^{-1}$. Therefore, setting $f' = a\bar{f}a^{-1}$, $g' = a\bar{g}a^{-1}$, and $h' = a\bar{h}a^{-1}$, we have proved the lemma in the desired special case.

The general case and the final claim can be obtained by the same argument as in the last paragraph of the proof of Lemma~\ref{open f}, so we omit it here, to minimize repetition.
\end{proof}

\section{Forward cycles} \label{fwd section}

In this section we shall show how to add forward cycles to the maps $f$ and $g$ in our equation $fg = h$. Since the numbers of forward cycles in the cycle decompositions $f$ and $g$ determine the number of forward cycles in the cycle decomposition of $h$, by Lemmas~\ref{fwd cycles} and~\ref{coimage}, no separate treatment of the latter is needed.  Let us begin with $f$.

\begin{lemma} \label{fwd f}
Let $f, g, h \in \Inj(\Omega)$ be any three maps, each having at least one infinite cycle, such that $fg = h$, and let $K$ be a cardinal satisfying $\, 1 \leq K \leq \aleph_0$. Further, suppose that there are infinite cycles $\, \Gamma_f$, $\, \Gamma_g$, and $\, \Gamma_h$ of $f$, $g$, and $h$, respectively, such that $\, \Sigma =  \Gamma_f \cap  \Gamma_g \cap  \Gamma_h$ is infinite. Then there exist maps  $f', g', h' \in \Inj(\Omega)$ having the following properties: 
\begin{enumerate}
\item[$(1)$] $f'g' = h';$ 
\item[$(2)$] $g' \sim g;$ 
\item[$(3)$] $(f')\C_{\mathrm{fwd}} = (f)\C_{\mathrm{fwd}}  + K$, $(f')\C_{\mathrm{open}} = (f)\C_{\mathrm{open}}$, $(f')\C_n = (f)\C_n$ for all $n \in \Z_+;$
\item[$(4)$] $(h')\C_{\mathrm{fwd}} = (h)\C_{\mathrm{fwd}}  + K$, $(h')\C_{\mathrm{open}} = (h)\C_{\mathrm{open}}$, $(h')\C_n = (h)\C_n$ for all $n \in \Z_+.$
\end{enumerate}

Furthermore, the maps $f'$, $g'$, and $h'$ can be chosen so that there are infinite cycles $\, \Gamma_{f'}$, $\, \Gamma_{g'}$, and $\, \Gamma_{h'}$ of $f'$, $g'$, and $h'$, respectively, such that $\, \Gamma_{f'} \cap  \Gamma_{g'} \cap  \Gamma_{h'}$ is infinite.
\end{lemma}

\begin{proof}
Let us begin by proving the lemma in the case where $K = 1$. 

Write $\Sigma = \{\alpha_i\}_{i\in \Z_+} \cup \Delta$, where the union is disjoint and $|\Delta| = \aleph_0$. We may choose the elements $\alpha_i$ so that none of them is initial in $\Gamma_g$ and $\Gamma_h$ (if either of them is a forward cycle). Similarly, we may assume that for all $\alpha_i$ and $\alpha_j$ we have $\alpha_j \notin \{(\alpha_i)g^{-1}, (\alpha_i)h^{-1}\}$. For each $i \in \Z_+$, write $\delta_i = (\alpha_i)g^{-1}$ and $\zeta_i  = (\alpha_i)h^{-1}$. Also, let $\Omega' = \Omega \dot{\cup} \{\theta_i\}_{i\in \Z_+}$. We shall next define injective endomaps $\bar{f}$, $\bar{g}$, and $\bar{h}$ on $\Omega'$.

Define $\bar{f}$ to act by 
$$\left\{ \begin{array}{ll}
\theta_i \mapsto \theta_{i+1} & \textrm{for } i\in \Z_+\\
\lambda \mapsto (\lambda)f & \textrm{for } \lambda \in \Omega
\end{array}\right.,$$ and define $\bar{g}$ to act by 
$$\left\{ \begin{array}{ll}
\delta_i \mapsto \theta_{2i-1} & \textrm{for } i\in \Z_+\\
\theta_{2i-1} \mapsto \theta_{2i} & \textrm{for } i\in \Z_+\\
\theta_{2i} \mapsto \alpha_{i} & \textrm{for } i\in \Z_+\\
\lambda \mapsto (\lambda)g & \textrm{otherwise}
\end{array}\right..$$

Thus, $\{\theta_i\}_{i\in \Z_+}$ is a forward cycle under $\bar{f}$, and otherwise $\bar{f}$ has the same cycles as $f$. It follows that for any bijection $a : \Omega \rightarrow \Omega'$, we have $(a\bar{f}a^{-1})\C_{\mathrm{fwd}} = (f)\C_{\mathrm{fwd}} + 1$, $(a\bar{f}a^{-1})\C_{\mathrm{open}} = (f)\C_{\mathrm{open}} $, and $(a\bar{f}a^{-1})\C_n = (f)\C_n$ for all $n \in \Z_+$, by Lemma~\ref{pass}. For each $i \in \Z_+$, the action of $\bar{g}$ on the elements $\alpha_i$, $\delta_i$, and $\theta_i$ can be visualized as follows: $$\cdots \rightarrow \delta_i \rightarrow \theta_{2i-1} \rightarrow \theta_{2i} \rightarrow \alpha_i \rightarrow \cdots .$$ Thus, $\Gamma_g \cup \{\theta_i\}_{i\in \Z_+}$ is an infinite cycle under $\bar{g}$, and otherwise $\bar{g}$ has the same cycles as $g$. It follows that for any bijection $a : \Omega \rightarrow \Omega'$, we have $g \sim a\bar{g}a^{-1}$.

Setting $\bar{h} = \bar{f}\bar{g}$, we see that $\bar{h}$ acts by 
$$\left\{ \begin{array}{ll}
\zeta_i \mapsto \theta_{2i-1} & \textrm{for } i\in \Z_+\\
\theta_{2i-1} \mapsto \alpha_{i} & \textrm{for } i\in \Z_+\\
\theta_{2i} \mapsto \theta_{2i+2} & \textrm{for } i\in \Z_+\\
\lambda \mapsto (\lambda)h & \textrm{otherwise}
\end{array}\right..$$
For each $i \in \Z_+$, the action of $\bar{h}$ on the elements $\alpha_i$, $\zeta_i$, and $\theta_{2i-1}$ can be visualized as follows: $$\cdots \rightarrow \zeta_i \rightarrow \theta_{2i-1} \rightarrow \alpha_i \rightarrow \cdots .$$ Thus, $\Gamma_h \cup \{\theta_{2i-1}\}_{i\in \Z}$ is an infinite cycle under $\bar{h}$, $\{\theta_{2i}\}_{i\in \Z_+}$ is a forward cycle under $\bar{h}$, and otherwise $\bar{h}$ has the same cycles as $h$. It follows that for any bijection $a : \Omega \rightarrow \Omega'$, we have $(a\bar{h}a^{-1})\C_{\mathrm{fwd}} = (h)\C_{\mathrm{fwd}} + 1$, $(a\bar{h}a^{-1})\C_{\mathrm{open}} = (h)\C_{\mathrm{open}} $, and $(a\bar{h}a^{-1})\C_n = (h)\C_n$ for all $n \in \Z_+.$ Therefore, setting $f' = a\bar{f}a^{-1}$, $g' = a\bar{g}a^{-1}$, and $h' = a\bar{h}a^{-1}$, we have proved the lemma in the desired special case.

The general case and the final claim can be obtained by the same argument as in the last paragraph of the proof of Lemma~\ref{open f}, so we omit it here.
\end{proof}

We conclude the section with an analogue of the previous lemma for $g$.

\begin{lemma} \label{fwd g}
Let $f, g, h \in \Inj(\Omega)$ be any three maps, each having at least one infinite cycle, such that $fg = h$, and let $K$ be a cardinal satisfying $\, 1 \leq K \leq \aleph_0$. Further, suppose that there are infinite cycles $\, \Gamma_f$, $\, \Gamma_g$, and $\, \Gamma_h$ of $f$, $g$, and $h$, respectively, such that $\, \Sigma =  \Gamma_f \cap  \Gamma_g \cap  \Gamma_h$ is infinite. Then there exist maps  $f', g', h' \in \Inj(\Omega)$ having the following properties: 
\begin{enumerate}
\item[$(1)$] $f'g' = h';$ 
\item[$(2)$] $f' \sim f;$ 
\item[$(3)$] $(g')\C_{\mathrm{fwd}} = (g)\C_{\mathrm{fwd}}  + K$, $(g')\C_{\mathrm{open}} = (g)\C_{\mathrm{open}}$, $(g')\C_n = (g)\C_n$ for all $n \in \Z_+;$
\item[$(4)$] $(h')\C_{\mathrm{fwd}} = (h)\C_{\mathrm{fwd}}  + K$, $(h')\C_{\mathrm{open}} = (h)\C_{\mathrm{open}}$, $(h')\C_n = (h)\C_n$ for all $n \in \Z_+.$
\end{enumerate}

Furthermore, the maps $f'$, $g'$, and $h'$ can be chosen so that there are infinite cycles $\, \Gamma_{f'}$, $\, \Gamma_{g'}$, and $\, \Gamma_{h'}$ of $f'$, $g'$, and $h'$, respectively, such that $\, \Gamma_{f'} \cap  \Gamma_{g'} \cap  \Gamma_{h'}$ is infinite.
\end{lemma}

\begin{proof}
As usual, let us begin by proving the lemma in the case where $K = 1$. 

Write $\Sigma = \{\alpha_i\}_{i\in \Z_+} \cup \Delta$, where the union is disjoint and $|\Delta| = \aleph_0$. We may choose these elements so that for all $\alpha_i$ and $\alpha_j$ we have $\alpha_j \notin \{(\alpha_i)f, (\alpha_i)h\}$. For each $i \in \Z_+$, write $\gamma_i = (\alpha_i)f$ and $\eta_i  = (\alpha_i)h$. Also, let $\Omega' = \Omega \dot{\cup} \{\theta_i\}_{i\in \Z_+}$. We shall next define injective endomaps $\bar{f}$, $\bar{g}$, and $\bar{h}$ on $\Omega'$.

Define $\bar{f}$ to act by 
$$\left\{ \begin{array}{ll}
\alpha_i \mapsto \theta_{2i-1} & \textrm{for } i\in \Z_+\\
\theta_{2i-1} \mapsto \theta_{2i} & \textrm{for } i\in \Z_+\\
\theta_{2i} \mapsto \gamma_{i} & \textrm{for } i\in \Z_+\\
\lambda \mapsto (\lambda)f & \textrm{otherwise}
\end{array}\right.,$$ and define $\bar{g}$ to act by 
$$\left\{ \begin{array}{ll}
\theta_i \mapsto \theta_{i+1} & \textrm{for } i\in \Z_+\\
\lambda \mapsto (\lambda)g & \textrm{for } \lambda \in \Omega
\end{array}\right..$$

For each $i \in \Z_+$, the action of $\bar{f}$ on the elements $\alpha_i$, $\gamma_i$, and $\theta_i$ can be visualized as follows: $$\cdots \rightarrow \alpha_i \rightarrow \theta_{2i-1} \rightarrow \theta_{2i} \rightarrow \gamma_i \rightarrow \cdots .$$ Thus, $\Gamma_f \cup \{\theta_i\}_{i\in \Z_+}$ is an infinite cycle under $\bar{f}$, and otherwise $\bar{f}$ has the same cycles as $f$. It follows that for any bijection $a : \Omega \rightarrow \Omega'$, we have $f \sim a\bar{f}a^{-1}$. Also, $\{\theta_i\}_{i\in \Z_+}$ is a forward cycle under $\bar{g}$, and otherwise $\bar{g}$ has the same cycles as $g$. Therefore, for any bijection $a : \Omega \rightarrow \Omega'$, we have $(a\bar{g}a^{-1})\C_{\mathrm{fwd}} = (g)\C_{\mathrm{fwd}} + 1$, $(a\bar{g}a^{-1})\C_{\mathrm{open}} = (g)\C_{\mathrm{open}}$, and $(a\bar{g}a^{-1})\C_n = (g)\C_n$ for all $n \in \Z_+.$

Setting $\bar{h} = \bar{f}\bar{g}$, we see that $\bar{h}$ acts by 
$$\left\{ \begin{array}{ll}
\alpha_i \mapsto \theta_{2i} & \textrm{for } i\in \Z_+\\
\theta_{2i-1} \mapsto \theta_{2i+1} & \textrm{for } i\in \Z_+\\
\theta_{2i} \mapsto \eta_{i} & \textrm{for } i\in \Z_+\\
\lambda \mapsto (\lambda)h & \textrm{otherwise}
\end{array}\right..$$ For each $i \in \Z_+$, the action of $\bar{h}$ on the elements $\alpha_i$, $\eta_i$, and $\theta_{2i}$ can be visualized as follows: $$\cdots \rightarrow \alpha_i \rightarrow \theta_{2i} \rightarrow \eta_i \rightarrow \cdots .$$ Thus, $\Gamma_h \cup \{\theta_{2i}\}_{i\in \Z}$ is an infinite cycle under $\bar{h}$, $\{\theta_{2i-1}\}_{i\in \Z_+}$ is a forward cycle under $\bar{h}$, and otherwise $\bar{h}$ has the same cycles as $h$. It follows that for any bijection $a : \Omega \rightarrow \Omega'$, we have $(a\bar{h}a^{-1})\C_{\mathrm{fwd}} = (h)\C_{\mathrm{fwd}} + 1$, $(a\bar{h}a^{-1})\C_{\mathrm{open}} = (h)\C_{\mathrm{open}} $, and $(a\bar{h}a^{-1})\C_n = (h)\C_n$ for all $n \in \Z_+.$ Therefore, setting $f' = a\bar{f}a^{-1}$, $g' = a\bar{g}a^{-1}$, and $h' = a\bar{h}a^{-1}$, we have proved the lemma in the desired special case.

The general case and the final claim can be obtained by the same argument as in the last paragraph of the proof of Lemma~\ref{open f}, so we omit it here.
\end{proof}

\section{The main constructions} \label{construction section}

We now turn to providing our ``anchor" constructions of maps $f$ and $g$, such that both have just one infinite cycle, while $fg$ has an arbitrary number of open cycles. To prove Theorem~\ref{comp} we shall then use the lemmas of the previous three sections to add other sorts of cycles to these maps.
 
\begin{lemma} \label{open h}
Let $K$ be a cardinal such that $\, 0 \leq K \leq \aleph_0$. Then there exist maps $f, g, h \in \Inj(\Omega)$ satisfying:
\begin{enumerate}
\item[$(1)$] $fg = h;$ 
\item[$(2)$] $(f)\C_{\mathrm{fwd}} = 1$, $(f)\C_{\mathrm{open}} = 0$, $(f)\C_n = 0$ for all $n \in \Z_+;$
\item[$(3)$]$(g)\C_{\mathrm{fwd}} = 1$, $(g)\C_{\mathrm{open}} = 0$, $(g)\C_n = 0$ for all $n \in \Z_+;$
\item[$(4)$] $(h)\C_{\mathrm{fwd}} = 2$, $(h)\C_{\mathrm{open}} = K$, $(h)\C_n = 0$ for all $n \in \Z_+.$
\end{enumerate}
\end{lemma}

\begin{proof}
If $K = 0$, then we can identify $\Omega$ with $\N$, and take both $f$ and $g$ to be the right shift map, defined by $i \rightarrow i + 1$ $(i \in \N)$. In this case, $\N$ is a forward cycle under both $f$ and $g$, and $(i)fg = i + 2$ for all $i \in \N$. Hence the cycle decomposition of $h = fg$ consists for the two forward cycles $\{2i : i \in \N\}$ and $\{2i + 1 : i \in \N\}$, as desired.

Let us therefore assume that $K \geq 1$ and now identify $\Omega$ with $\{0,1, \dots, K\} \times \Z$ (or with $\N \times \Z$, if $K$ is infinite). We define $f \in \Inj(\Omega)$ to act by 
$$\left\{ \begin{array}{ll}
(0,j) \mapsto (0,-j) & \textrm{for } j > 0\\
(0,j) \mapsto (1,j) & \textrm{for } j \leq 0\\
(1,j) \mapsto (0,j+1) & \textrm{for } j \geq 0\\
(i,j) \mapsto (i-1,j+2) & \textrm{for } i>1, j \geq 0\\
(i,-1) \mapsto (i,1) & \textrm{for } i>0\\
(i,j) \mapsto (i+1,j+2) & \textrm{for } K> i > 0, j<-1\\
(K,j) \mapsto (K,-j) & \textrm{for } j<-1 \textrm{ (if } $K$ \textrm{ is finite)}\\
\end{array}\right.,$$ and we define $g \in \Inj(\Omega)$ to act by
$$\left\{ \begin{array}{ll}
(0,j) \mapsto (1,j) & \textrm{for } j > 0\\
(0,j) \mapsto (0,-j+1) & \textrm{for } j \leq 0\\
(i,1) \mapsto (i,0) & \textrm{for } i>0\\
(i,j) \mapsto (i-1,j-1) & \textrm{for } i>0, j\leq 0\\
(i,j) \mapsto (i+1,j-1) & \textrm{for } K> i>0, j > 1\\
(K,j) \mapsto (K,-j+1) & \textrm{for } j>1 \textrm{ (if } $K$ \textrm{ is finite)}\\
\end{array}\right..$$ Then, setting $fg = h$, it is routine to check that $h$ acts by 
$$\left\{ \begin{array}{ll}
(0,j) \mapsto (0,j+1) & \textrm{for } j > 0\\
(0,j) \mapsto (0,j-1) & \textrm{for } j \leq 0\\
(i,j) \mapsto (i,j+1) & \textrm{for } i>0, j \in \Z\\
\end{array}\right..$$ 

Thus, for each $i > 0,$ $\{(i,j) : j \in \Z\}$ is an open cycle under $h$, and $\{(0,j) : j \leq 0\}$ and $\{(0,j) : j > 0\}$ are forward cycles under $h$, showing that $h$ has the desired form. It remains to verify that $f$ and $g$, as defined above, both have exactly one forward cycle and no other cycles. This is immediate once $f$ and $g$ are represented as graphs - see Figures 1 and 2 in the Appendix below. In the graphs one can see that both $f$ and $g$ have just one (forward) cycle, starting at the point $(0,0)$. These graphs represent the maps in the case where $K$ is finite. If $K$ were infinite, then the graphs of $f$ and $g$ would be just like those in Figures 1 and 2, except without the last two rows.
\end{proof}

The next lemma is an analogue of the previous one, except now $g$ will have an open cycle rather than a forward cycle.

\begin{lemma} \label{open h 2}
Let $K$ be a cardinal such that $\, 0 \leq K \leq \aleph_0$. Then there exist maps $f, g, h \in \Inj(\Omega)$ satisfying:
\begin{enumerate}
\item[$(1)$] $fg = h;$ 
\item[$(2)$] $(f)\C_{\mathrm{fwd}} = 1$, $(f)\C_{\mathrm{open}} = 0$, $(f)\C_n = 0$ for all $n \in \Z_+;$
\item[$(3)$]$(g)\C_{\mathrm{fwd}} = 0$, $(g)\C_{\mathrm{open}} = 1$, $(g)\C_n = 0$ for all $n \in \Z_+;$
\item[$(4)$] $(h)\C_{\mathrm{fwd}} = 1$, $(h)\C_{\mathrm{open}} = K$, $(h)\C_n = 0$ for all $n \in \Z_+.$
\end{enumerate}
\end{lemma}

\begin{proof}
First suppose that $K = 0$. Let us indentify $\Omega$ with $\Z$ and define
$$(i) f =\left\{ \begin{array}{ll}
-i & \textrm{for } i > 0\\
-i+1 & \textrm{for } i \leq 0\\
\end{array}\right.,$$ $(i) g = i+1$ $(i \in \Z)$. Then, setting $h = fg$, we have 
$$(i) h =\left\{ \begin{array}{cl}
-i+1 & \textrm{for } i > 0\\
-i+2 & \textrm{for } i \leq 0\\
\end{array}\right..$$ The action of $f$ on $\Omega$ can be visualized as follows: $$0 \rightarrow 1 \rightarrow -1 \rightarrow 2 \rightarrow -2 \rightarrow \cdots.$$ Also, action of $h$ on $\Omega$ can be visualized as follows:
$$1 \rightarrow \fbox{$0$} \rightarrow 2 \rightarrow \fbox{$-1$} \rightarrow 3 \rightarrow \fbox{$-2$} \rightarrow \cdots.$$ Thus, $f$, $g$, and $h$ have the desired forms.

Now suppose that $K = 1$. Let us write $\Omega = \{\alpha_i : i \in \Z\} \cup \{\beta_i : i \in \Z\}$, define $f$ to act by
$$\left\{ \begin{array}{ll}
\beta_i \mapsto \alpha_{i+1} & \textrm{for } i \in \Z\\
\alpha_i \mapsto \beta_{-i} & \textrm{for } i \leq 0 \\
\alpha_i \mapsto \beta_{-i-1} & \textrm{for } i > 0\\
\end{array}\right.,$$ and define $g$ to act by 
$$\left\{ \begin{array}{ll}
\beta_i \mapsto \alpha_{i+1} & \textrm{for } i \in \Z\\
\alpha_i \mapsto \beta_{i} & \textrm{for } i \in \Z\\
\end{array}\right..$$ The action of $f$ on $\Omega$ can be visualized as follows: $$\beta_{-1} \rightarrow \alpha_{0} \rightarrow \beta_{0} \rightarrow \alpha_{1} \rightarrow$$ $$\beta_{-2} \rightarrow \alpha_{-1} \rightarrow \beta_{1} \rightarrow \alpha_{2} \rightarrow $$ $$\vdots$$
$$\beta_{-i} \rightarrow \alpha_{-i+1} \rightarrow \beta_{i-1} \rightarrow \alpha_{i} \rightarrow \cdots.$$ Hence, $f$ has exactly one forward cycle and no other cycles. It is also clear that $g$ has exactly one open cycle and no other cycles. Setting $h = fg$, see that $h$ acts by 
$$\left\{ \begin{array}{ll}
\beta_i \mapsto \beta_{i+1} & \textrm{for } i \in \Z\\
\alpha_i \mapsto \alpha_{-i+1} & \textrm{for } i \leq 0 \\
\alpha_i \mapsto \alpha_{-i} & \textrm{for } i > 0\\
\end{array}\right..$$ The action of $h$ on $\{\alpha_i : i \in \Z\}$ can be visualized as follows: $$\alpha_{0} \rightarrow \alpha_{1} \rightarrow \alpha_{-1} \rightarrow \alpha_{2} \rightarrow \alpha_{-2} \rightarrow \cdots.$$ Thus, $\{\alpha_i : i \in \Z\}$ is a forward cycle under $h$ and $\{\beta_i : i \in \Z\}$ is an open cycle under $h$, giving $h$ the desired form.

Let us therefore assume that $K \geq 2$. This time write $$\Omega = (\{i \in \N : i < K\} \times \Z) \cup \{(-1,j) : j \in \Z_+\} \subseteq \Z \times \Z.$$ We define $f \in \Inj(\Omega)$ to act by 
$$\left\{ \begin{array}{ll}
(-1,2n) \mapsto (0,-2n+1) & \textrm{for } n>0\\
(-1,2n-1) \mapsto (0,-2n) & \textrm{for } n>0\\
(0,2n) \mapsto (-1,2n-1) & \textrm{for } n>0\\
(0,2n-1) \mapsto (-1,2n) & \textrm{for } n>0\\
(0,j) \mapsto (1,j) & \textrm{for } j \leq 0\\
(1,j) \mapsto (0,j+1) & \textrm{for } j \geq 0\\
(i,j) \mapsto (i-1,j+2) & \textrm{for } i>1, j \geq 0\\
(i,-1) \mapsto (i,1) & \textrm{for } i>0\\
(i,j) \mapsto (i+1,j+2) & \textrm{for } K-1> i > 0, j<-1\\
(K-1,j) \mapsto (K-1,-j) & \textrm{for } j<-1 \textrm{ (if } $K$ \textrm{ is finite)}\\
\end{array}\right.,$$ and we define $g \in \Inj(\Omega)$ to act by
$$\left\{ \begin{array}{ll}
(-1,2n) \mapsto (0,2n-2) & \textrm{for } n>0\\
(-1,2n-1) \mapsto (0,2n-1) & \textrm{for } n>0\\
(0,j) \mapsto (1,j) & \textrm{for } j > 0\\
(0,0) \mapsto (-1,1) & \\
(0,2n) \mapsto (-1,-2n) & \textrm{for } n<0\\
(0,2n+1) \mapsto (-1,-2n+1) & \textrm{for } n<0\\
(i,1) \mapsto (i,0) & \textrm{for } i>0\\
(i,j) \mapsto (i-1,j-1) & \textrm{for } i>0, j\leq 0\\
(i,j) \mapsto (i+1,j-1) & \textrm{for } K-1> i>0, j > 1\\
(K-1,j) \mapsto (K-1,-j+1) & \textrm{for } j>1 \textrm{ (if } $K$ \textrm{ is finite)}\\
\end{array}\right..$$ We note that on elements $(i,j)$ where $i > 0$, $f$ and $g$ are defined exactly as in Lemma~\ref{open h}. Setting $fg = h$, we see that $h$ acts by
$$\left\{ \begin{array}{ll}
(-1,j) \mapsto (-1,j+1) & \textrm{for } j>0\\
(0,j) \mapsto (0,j-1) & \textrm{for } j \in \Z\\
(i,j) \mapsto (i,j+1) & \textrm{for } i>0, j \in \Z\\
\end{array}\right..$$

Thus, for each $i \geq 0,$ $\{(i,j) : j \in \Z\}$ is an open cycle under $h$, and $\{(-1,j) : j \in \Z_+\}$ is a forward cycle under $h$, showing that $h$ has the desired form. It remains to verify that $f$ and $g$ have appropriate shapes. As in the previous lemma, this is immediate once $f$ and $g$ are represented as graphs -- see Figures 3 and 4 in the Appendix below. In the graphs one can see that $f$ has just a forward cycle, starting at the point $(0,0)$, and that $g$ has just an open cycle. These graphs represent the maps in the case where $K$ is finite. If $K$ were infinite, then the graphs of $f$ and $g$ would be the same, except without the last two rows.
\end{proof}

It might appear that the $f$ and $g$ above are defined in a manner that is more complicated than necessary, producing the crossed arrows in the graphs. We have chosen these maps this way in order to give $h$ a very simple form and to reuse as much of the construction from Lemma~\ref{open h} as possible.

The roles of $f$ and $g$ in the previous lemma are interchangeable, as the next corollary shows.

\begin{corollary} \label{open h 2.5}
Let $K$ be a cardinal such that $\, 0 \leq K \leq \aleph_0$. Then there exist maps $f, g, h \in \Inj(\Omega)$ satisfying:
\begin{enumerate}
\item[$(1)$] $fg = h;$ 
\item[$(2)$] $(f)\C_{\mathrm{fwd}} = 0$, $(f)\C_{\mathrm{open}} = 1$, $(f)\C_n = 0$ for all $n \in \Z_+;$
\item[$(3)$]$(g)\C_{\mathrm{fwd}} = 1$, $(g)\C_{\mathrm{open}} = 0$, $(g)\C_n = 0$ for all $n \in \Z_+;$
\item[$(4)$] $(h)\C_{\mathrm{fwd}} = 1$, $(h)\C_{\mathrm{open}} = K$, $(h)\C_n = 0$ for all $n \in \Z_+.$
\end{enumerate}
\end{corollary}

\begin{proof}
By Lemma~\ref{open h 2} there are maps $f', g', h' \in \Inj(\Omega)$ such that 
\begin{enumerate}
\item[$(1)$] $f'g' = h';$ 
\item[$(2)$] $(f')\C_{\mathrm{fwd}} = 1$, $(f')\C_{\mathrm{open}} = 0$, $(f')\C_n = 0$ for all $n \in \Z_+;$
\item[$(3)$]$(g')\C_{\mathrm{fwd}} = 0$, $(g')\C_{\mathrm{open}} = 1$, $(g')\C_n = 0$ for all $n \in \Z_+;$
\item[$(4)$] $(h')\C_{\mathrm{fwd}} = 1$, $(h')\C_{\mathrm{open}} = K$, $(h')\C_n = 0$ for all $n \in \Z_+.$
\end{enumerate}
Since $g' \in \Sym(\Omega)$, we have $g'h'(g')^{-1} = g'f'g'(g')^{-1} = g'f'$. Setting $h = g'h'(g')^{-1}$, $f = g'$, and $g = f'$, we obtain maps with the desired properties.
\end{proof}

\section{Putting it all together} \label{proof section}

We are now ready to give a complete proof of Theorem~\ref{comp}. For the ``if" direction, this proof essentially says that we can start with the maps defined in one of Lemma~\ref{open h}, Lemma~\ref{open h 2}, or Corollary~\ref{open h 2.5}, and then add to them enough finite, open, and forward cycles, using the lemmas of the three earlier sections, to construct any maps $f, g, h \in \Inj(\Omega)$ that each have at least one infinite cycle, and satisfy $fg = h$ and $|\Omega \setminus (\Omega)f| + |\Omega \setminus (\Omega)g| = |\Omega \setminus (\Omega)h|$. For the convenience of the reader, some of the remarks made in Section~\ref{plan section} will be repeated below.

\begin{proof}[Proof of Theorem~\ref{comp}]
Let $f, g, h \in \Inj(\Omega)$ be any three maps. Suppose that there exist permutations $a, b \in \Sym(\Omega)$ such that $h = afa^{-1}bgb^{-1}$. By Proposition~\ref{decomposition} and Lemma~\ref{fwd cycles}, we have $|\Omega \setminus (\Omega)f| = |\Omega \setminus (\Omega)afa^{-1}|$ and $|\Omega \setminus (\Omega)g| = |\Omega \setminus (\Omega)aga^{-1}|$. Hence, it follows from Lemma~\ref{coimage} that $|\Omega \setminus (\Omega)f| + |\Omega \setminus (\Omega)g| = |\Omega \setminus (\Omega)h|$.

Conversely, suppose that $f, g, h \in \Inj(\Omega)$ all have at least one infinite cycle and satisfy $|\Omega \setminus (\Omega)f| + |\Omega \setminus (\Omega)g| = |\Omega \setminus (\Omega)h|$.  By Lemma~\ref{prod-conj-equiv}, to show that there exist permutations $a, b \in \Sym(\Omega)$ such that $h = afa^{-1}bgb^{-1}$, it suffices to construct elements $f', g' \in  \Inj(\Omega)$ such that $f' \sim f$, $g' \sim g$, and $f'g' \sim h$.

We shall ignore the cases where $f,g \in \Sym(\Omega)$, in view of Theorem~\ref{droste theorem}. For the other cases, we shall construct $f'$ and $g'$ in stages, beginning with maps $f_1, g_1 \in \Inj(\Omega)$. Now, set $K = (h)\C_{\mathrm{open}}$. If $f, g \in \Inj(\Omega) \setminus \Sym(\Omega)$, then we let $f_1$ and $g_1$ be as in Lemma~\ref{open h}. If $f \in \Inj(\Omega) \setminus \Sym(\Omega)$ but $g \in \Sym(\Omega)$, then we let $f_1$ and $g_1$ be as in Lemma~\ref{open h 2}. Finally, if $f \in \Sym(\Omega)$ but $g \in \Inj(\Omega) \setminus \Sym(\Omega)$, then we let $f_1$ and $g_1$ be as in Corollary~\ref{open h 2.5}.

In all cases, both $f_1$ and $g_1$ have exactly one cycle in their cycle decompositions. Let us denote those cycles by $\Gamma_{f_1}$ and $\Gamma_{g_1}$, respectively. Since, in each case, every cycle in the cycle decomposition of $f_1g_1$ is infinite, letting $\Gamma_{f_1g_1}$ denote any such cycle, we have $|\Gamma_{f_1} \cap \Gamma_{g_1} \cap \Gamma_{f_1g_1}| = \aleph_0$. Since $(f_1g_1)\C_{\mathrm{open}} = (h)\C_{\mathrm{open}}$, by Lemma~\ref{fin h} we can find maps $f_2, g_2, h_2 \in \Inj(\Omega)$ such that 
\begin{enumerate}
\item[$(1)$] $f_2g_2 = h_2$;
\item[$(2)$] $f_2 \sim f_1$; 
\item[$(3)$] $g_2 \sim g_1$;
\item[$(4)$] $(h_2)\C_{\mathrm{fwd}} = (h_1)\C_{\mathrm{fwd}}$, $(h_2)\C_{\mathrm{open}} = (h)\C_{\mathrm{open}}$, $(h_2)\C_n = (h)\C_n$ for all $n \in \Z_+$;
\item[$(5)$] there are infinite cycles $\Gamma_{f_2}$, $\Gamma_{g_2}$, and $\Gamma_{h_2}$ of $f_2$, $g_2$, and $h_2$, respectively, such that $|\Gamma_{f_2} \cap  \Gamma_{g_2} \cap  \Gamma_{h_2}| = \aleph_0$.
\end{enumerate} 
Applying Lemma~\ref{fwd f} to $f_2$, $g_2$, and $h_2$, we can find maps $f_3, g_3, h_3 \in \Inj(\Omega)$ such that 
\begin{enumerate}
\item[$(1)$] $f_3g_3 = h_3$;
\item[$(2)$] $(f_3)\C_{\mathrm{fwd}} = (f)\C_{\mathrm{fwd}}$, $(f_3)\C_{\mathrm{open}} = (f_1)\C_{\mathrm{open}}$, $(f_3)\C_n = 0$ for all $n \in \Z_+$;
\item[$(3)$] $g_3 \sim g_1$;
\item[$(4)$] $(h_3)\C_{\mathrm{fwd}} = (f)\C_{\mathrm{fwd}} + (g_1)\C_{\mathrm{fwd}}$, $(h_3)\C_{\mathrm{open}} = (h)\C_{\mathrm{open}}$, $(h_3)\C_n = (h)\C_n$ for all $n \in \Z_+$;
\item[$(5)$] there are infinite cycles $\Gamma_{f_3}$, $\Gamma_{g_3}$, and $\Gamma_{h_3}$ of $f_3$, $g_3$, and $h_3$, respectively, such that $|\Gamma_{f_3} \cap  \Gamma_{g_3} \cap  \Gamma_{h_3}| = \aleph_0$.
\end{enumerate} 
Applying Lemma~\ref{fwd g}  (and the observation following Definition~\ref{cycle number}) to $f_3$, $g_3$, and $h_3$, we can find maps $f_4, g_4, h_4 \in \Inj(\Omega)$ such that 
\begin{enumerate}
\item[$(1)$] $f_4g_4 = h_4$;
\item[$(2)$] $f_4 \sim f_3$;
\item[$(3)$] $(g_4)\C_{\mathrm{fwd}} = (g)\C_{\mathrm{fwd}}$, $(g_4)\C_{\mathrm{open}} = (g_1)\C_{\mathrm{open}}$, $(g_4)\C_n = 0$ for all $n \in \Z_+$;
\item[$(4)$] $h_4 \sim h$;
\item[$(5)$] there are infinite cycles $\Gamma_{f_4}$, $\Gamma_{g_4}$, and $\Gamma_{h_4}$ of $f_4$, $g_4$, and $h_4$, respectively, such that $|\Gamma_{f_4} \cap  \Gamma_{g_4} \cap  \Gamma_{h_4}| = \aleph_0$.
\end{enumerate} 
Applying Lemma~\ref{fin f} to $f_4$, $g_4$, and $h_4$, we can find maps $f_5, g_5, h_5 \in \Inj(\Omega)$ such that 
\begin{enumerate}
\item[$(1)$] $f_5g_5 = h_5$;
\item[$(2)$] $(f_5)\C_{\mathrm{fwd}} = (f)\C_{\mathrm{fwd}}$, $(f_5)\C_{\mathrm{open}} = (f_1)\C_{\mathrm{open}}$, $(f_5)\C_n = (f)\C_n$ for all $n \in \Z_+$;
\item[$(3)$] $g_5 \sim g_4$;
\item[$(4)$] $h_5 \sim h$;
\item[$(5)$] there are infinite cycles $\Gamma_{f_5}$, $\Gamma_{g_5}$, and $\Gamma_{h_5}$ of $f_5$, $g_5$, and $h_5$, respectively, such that $|\Gamma_{f_5} \cap  \Gamma_{g_5} \cap  \Gamma_{h_5}| = \aleph_0$.
\end{enumerate} 
Applying Lemma~\ref{open f} to $f_5$, $g_5$, and $h_5$, we can find maps $f_6, g_6, h_6 \in \Inj(\Omega)$ such that 
\begin{enumerate}
\item[$(1)$] $f_6g_6 = h_6$;
\item[$(2)$] $f_6 \sim f$;
\item[$(3)$] $g_6 \sim g_5$;
\item[$(4)$] $h_6 \sim h$;
\item[$(5)$] there are infinite cycles $\Gamma_{f_6}$, $\Gamma_{g_6}$, and $\Gamma_{h_6}$ of $f_6$, $g_6$, and $h_6$, respectively, such that $|\Gamma_{f_6} \cap  \Gamma_{g_6} \cap  \Gamma_{h_6}| = \aleph_0$.
\end{enumerate} 
Applying Lemma~\ref{fin g} to $f_6$, $g_6$, and $h_6$, we can find maps $f_7, g_7, h_7 \in \Inj(\Omega)$ such that 
\begin{enumerate}
\item[$(1)$] $f_7g_7 = h_7$;
\item[$(2)$] $f_7 \sim f$;
\item[$(3)$] $(g_7)\C_{\mathrm{fwd}} = (g)\C_{\mathrm{fwd}}$, $(g_7)\C_{\mathrm{open}} = (g_1)\C_{\mathrm{open}}$, $(g_7)\C_n = (g)\C_n$ for all $n \in \Z_+$;
\item[$(4)$] $h_7 \sim h$;
\item[$(5)$] there are infinite cycles $\Gamma_{f_7}$, $\Gamma_{g_7}$, and $\Gamma_{h_7}$ of $f_7$, $g_7$, and $h_7$, respectively, such that $|\Gamma_{f_7} \cap  \Gamma_{g_7} \cap  \Gamma_{h_7}| = \aleph_0$.
\end{enumerate} 
Finally, applying Lemma~\ref{open g} to $f_7$, $g_7$, and $h_7$, we can find maps $f', g' \in \Inj(\Omega)$ such that $f' \sim f$, $g' \sim g$, and $f'g' \sim h$, as desired. 
\end{proof}

\newpage

\section*{Appendix: graphs}

\begin{picture}(360,250)(0,0)
\put(115,10){{\bf Figure 1.} Graph of $f$ in Lemma~\ref{open h}.}

\begin{tiny}
% 0 row
\put(195,230){$(0,0)$} \put(235,230){$(0,1)$}
\put(275,230){$(0,2)$} \put(315,230){$(0,3)\ \dots$}
\put(155,230){$(0,-1)$} \put(115,230){$(0,-2)$} 
\put(75,230){$(0,-3)$} \put(63,230){$\dots$}

% 1 row
\put(195,195){$(1,0)$} \put(235,195){$(1,1)$}
\put(275,195){$(1,2)$} \put(315,195){$(1,3)\ \dots$}
\put(155,195){$(1,-1)$} \put(115,195){$(1,-2)$} 
\put(75,195){$(1,-3)$} \put(63,195){$\dots$}

% 2 row
\put(195,160){$(2,0)$} \put(235,160){$(2,1)$}
\put(275,160){$(2,2)$} \put(315,160){$(2,3)\ \dots$}
\put(155,160){$(2,-1)$} \put(115,160){$(2,-2)$} 
\put(75,160){$(2,-3)$} \put(63,160){$\dots$}

% 3 row
\put(195,125){$(3,0)$} \put(235,125){$(3,1)$}
\put(275,125){$(3,2)$} \put(315,125){$(3,3)\ \dots$}
\put(155,125){$(3,-1)$} \put(115,125){$(3,-2)$} 
\put(75,125){$(3,-3)$} \put(63,125){$\dots$}

% dots
\put(203,100){$\vdots$} \put(243,100){$\vdots$}
\put(283,100){$\vdots$} \put(323,100){$\vdots$}
\put(163,100){$\vdots$} \put(123,100){$\vdots$} 
\put(83,100){$\vdots$}

% K-1 row
\put(190,80){$(K-1,0)$} \put(230,80){$(K-1,1)$}
\put(270,80){$(K-1,2)$} \put(310,80){$(K-1,3)\ \dots$}
\put(150,80){$(K-1,-1)$}
\put(110,80){$(K-1,-2)$} \put(70,80){$(K-1,-3)$}
\put(57,80){$\dots$}

% K row
\put(195,45){$(K,0)$} \put(235,45){$(K,1)$}
\put(275,45){$(K,2)$} \put(315,45){$(K,3)\ \dots$}
\put(155,45){$(K,-1)$} \put(115,45){$(K,-2)$} 
\put(75,45){$(K,-3)$} \put(63,45){$\dots$}

% ovals over 0 row
\put(207,235){\oval(75,20)[t]} \put(169,234){\vector(0,-1){0}}
\put(207,235){\oval(155,25)[t]} \put(129,234){\vector(0,-1){0}}
\put(207,235){\oval(235,30)[t]} \put(89,234){\vector(0,-1){0}}

% arrows between 0 row and 1 row
\put(203,228){\vector(0,-1){28}} \put(169,228){\vector(0,-1){28}}
\put(129,228){\vector(0,-1){28}} \put(89,228){\vector(0,-1){28}}
\put(205,201){\vector(4,3){37}} \put(245,201){\vector(4,3){37}}
\put(285,201){\vector(4,3){37}}

% arrows between 1 row and 2 row
\put(207,193){\oval(75,20)[b]} \put(244,194){\vector(0,1){0}}
\put(85,193){\vector(3,-1){79}} \put(125,193){\vector(3,-1){79}}
\put(207,167){\vector(3,1){79}} \put(247,167){\vector(3,1){79}} 

% arrows between 2 row and 3 row
\put(207,158){\oval(75,20)[b]} \put(244,159){\vector(0,1){0}}
\put(85,158){\vector(3,-1){79}} \put(125,158){\vector(3,-1){79}}
\put(207,132){\vector(3,1){79}} \put(247,132){\vector(3,1){79}} 

% arrows between K-1 row and K row
\put(207,78){\oval(75,20)[b]} \put(244,79){\vector(0,1){0}}
\put(85,78){\vector(3,-1){79}} \put(125,78){\vector(3,-1){79}}
\put(207,52){\vector(3,1){79}} \put(247,52){\vector(3,1){79}} 

% ovals below K row
\put(207,43){\oval(75,20)[b]} \put(244,44){\vector(0,1){0}}
\put(207,43){\oval(155,25)[b]} \put(284,44){\vector(0,1){0}}
\put(207,43){\oval(235,30)[b]} \put(324,44){\vector(0,1){0}}

\end{tiny}
\end{picture}

\vspace{.1in}

\begin{picture}(360,250)(0,0)
\put(115,10){{\bf Figure 2.} Graph of $g$ in Lemma~\ref{open h}.}

\begin{tiny}
% 0 row
\put(195,230){$(0,0)$} \put(235,230){$(0,1)$}
\put(275,230){$(0,2)$} \put(315,230){$(0,3)\ \dots$}
\put(155,230){$(0,-1)$} \put(115,230){$(0,-2)$} 
\put(75,230){$(0,-3)$} \put(63,230){$\dots$}

% 1 row
\put(195,195){$(1,0)$} \put(235,195){$(1,1)$}
\put(275,195){$(1,2)$} \put(315,195){$(1,3)\ \dots$}
\put(155,195){$(1,-1)$} \put(115,195){$(1,-2)$} 
\put(75,195){$(1,-3)$} \put(63,195){$\dots$}

% 2 row
\put(195,160){$(2,0)$} \put(235,160){$(2,1)$}
\put(275,160){$(2,2)$} \put(315,160){$(2,3)\ \dots$}
\put(155,160){$(2,-1)$} \put(115,160){$(2,-2)$} 
\put(75,160){$(2,-3)$} \put(63,160){$\dots$}

% 3 row
\put(195,125){$(3,0)$} \put(235,125){$(3,1)$}
\put(275,125){$(3,2)$} \put(315,125){$(3,3)\ \dots$}
\put(155,125){$(3,-1)$} \put(115,125){$(3,-2)$} 
\put(75,125){$(3,-3)$} \put(63,125){$\dots$}

% dots
\put(203,100){$\vdots$} \put(243,100){$\vdots$}
\put(283,100){$\vdots$} \put(323,100){$\vdots$}
\put(163,100){$\vdots$} \put(123,100){$\vdots$} 
\put(83,100){$\vdots$}

% K-1 row
\put(190,80){$(K-1,0)$} \put(230,80){$(K-1,1)$}
\put(270,80){$(K-1,2)$} \put(310,80){$(K-1,3)\ \dots$}
\put(150,80){$(K-1,-1)$}
\put(110,80){$(K-1,-2)$} \put(70,80){$(K-1,-3)$}
\put(57,80){$\dots$}

% K row
\put(195,45){$(K,0)$} \put(235,45){$(K,1)$}
\put(275,45){$(K,2)$} \put(315,45){$(K,3)\ \dots$}
\put(155,45){$(K,-1)$} \put(115,45){$(K,-2)$} 
\put(75,45){$(K,-3)$} \put(63,45){$\dots$}

% ovals over 0 row
\put(222,235){\oval(40,20)[t]} \put(242,234){\vector(0,-1){0}}
\put(222,235){\oval(120,25)[t]} \put(282,234){\vector(0,-1){0}}
\put(222,235){\oval(200,30)[t]} \put(322,234){\vector(0,-1){0}}

% arrows between 0 row and 1 row
\put(242,228){\vector(0,-1){28}} \put(282,228){\vector(0,-1){28}}
\put(322,228){\vector(0,-1){28}} \put(201,200){\vector(-4,3){37}} \put(161,200){\vector(-4,3){37}} \put(121,200){\vector(-4,3){37}} 

% arrows between 1 row and 2 row
\put(222,193){\oval(40,20)[b]} \put(202,194){\vector(0,1){0}}
\put(201,165){\vector(-4,3){37}} \put(161,165){\vector(-4,3){37}} \put(121,165){\vector(-4,3){37}} \put(282,193){\vector(-4,-3){37}}
\put(322,193){\vector(-4,-3){37}}

% arrows between 2 row and 3 row
\put(222,158){\oval(40,20)[b]} \put(202,159){\vector(0,1){0}}
\put(201,130){\vector(-4,3){37}} \put(161,130){\vector(-4,3){37}} \put(121,130){\vector(-4,3){37}} \put(282,158){\vector(-4,-3){37}}
\put(322,158){\vector(-4,-3){37}}

% arrows between K-1 row and K row
\put(222,78){\oval(40,20)[b]} \put(202,79){\vector(0,1){0}}
\put(201,50){\vector(-4,3){37}} \put(161,50){\vector(-4,3){37}} 
\put(121,50){\vector(-4,3){37}} \put(282,78){\vector(-4,-3){37}}
\put(322,78){\vector(-4,-3){37}}

% ovals below K row
\put(222,43){\oval(40,20)[b]} \put(202,44){\vector(0,1){0}}
\put(222,43){\oval(120,25)[b]} \put(162,44){\vector(0,1){0}}
\put(222,43){\oval(200,30)[b]} \put(122,44){\vector(0,1){0}}

\end{tiny}
\end{picture}

\vspace{.1in}

\begin{picture}(390,300)(0,0)
\put(115,10){{\bf Figure 3.} Graph of $f$ in Lemma~\ref{open h 2}.}

\begin{tiny}
% -1 row
\put(232,265){$(-1,1)$} \put(272,265){$(-1,2)$} 
\put(312,265){$(-1,3)$} \put(352,265){$(-1,4)\ \dots$}

% 0 row
\put(195,230){$(0,0)$} \put(235,230){$(0,1)$}
\put(275,230){$(0,2)$}
\put(155,230){$(0,-1)$} \put(115,230){$(0,-2)$} 
\put(75,230){$(0,-3)$} \put(35,230){$(0,-4)$} 
\put(23,230){$\dots$} \put(315,230){$(0,3)$}
\put(355,230){$(0,4)\ \dots$}

% 1 row
\put(195,195){$(1,0)$} \put(235,195){$(1,1)$}
\put(275,195){$(1,2)$}
\put(155,195){$(1,-1)$} \put(115,195){$(1,-2)$} 
\put(75,195){$(1,-3)$} \put(35,195){$(1,-4)$} 
\put(23,195){$\dots$} \put(315,195){$(1,3)$}
\put(355,195){$(1,4)\ \dots$}

% 2 row
\put(195,160){$(2,0)$} \put(235,160){$(2,1)$}
\put(275,160){$(2,2)$} 
\put(155,160){$(2,-1)$} \put(115,160){$(2,-2)$} 
\put(75,160){$(2,-3)$} \put(35,160){$(2,-4)$} 
\put(23,160){$\dots$} \put(315,160){$(2,3)$}
\put(355,160){$(2,4)\ \dots$}

% 3 row
\put(195,125){$(3,0)$} \put(235,125){$(3,1)$}
\put(275,125){$(3,2)$}
\put(155,125){$(3,-1)$} \put(115,125){$(3,-2)$} 
\put(75,125){$(3,-3)$} \put(35,125){$(3,-4)$} 
\put(23,125){$\dots$} \put(315,125){$(3,3)$}
\put(355,125){$(3,4)\ \dots$}

% dots
\put(203,100){$\vdots$} \put(243,100){$\vdots$}
\put(283,100){$\vdots$} \put(323,100){$\vdots$}
\put(163,100){$\vdots$} \put(123,100){$\vdots$} 
\put(83,100){$\vdots$} \put(43,100){$\vdots$}
\put(363,100){$\vdots$}

% K-2 row
\put(190,80){$(K-2,0)$} \put(230,80){$(K-2,1)$}
\put(270,80){$(K-2,2)$}
\put(150,80){$(K-2,-1)$}
\put(110,80){$(K-2,-2)$} \put(70,80){$(K-2,-3)$}
\put(30,80){$(K-2,-4)$} 
\put(17,80){$\dots$} \put(310,80){$(K-2,3)$}
\put(350,80){$(K-2,4)\ \dots$}

% K-1 row
\put(190,45){$(K-1,0)$} \put(230,45){$(K-1,1)$}
\put(270,45){$(K-1,2)$}
\put(150,45){$(K-1,-1)$}
\put(110,45){$(K-1,-2)$} \put(70,45){$(K-1,-3)$}
\put(30,45){$(K-1,-4)$} 
\put(17,45){$\dots$} \put(310,45){$(K-1,3)$}
\put(350,45){$(K-1,4)\ \dots$}

% arrows between -1 row and 0 row
\put(245,236){\vector(4,3){37}} \put(325,236){\vector(4,3){37}}
\put(285,236){\vector(-4,3){37}} \put(365,236){\vector(-4,3){37}}
\put(187,270){\oval(116,20)[t]} \put(129,270){\vector(0,-1){33}}
\put(227,270){\oval(116,30)[t]} \put(169,270){\vector(0,-1){33}}
\put(187,270){\oval(276,40)[t]} \put(49,270){\vector(0,-1){33}}
\put(227,270){\oval(276,50)[t]} \put(89,270){\vector(0,-1){33}}

% arrows between 0 row and 1 row
\put(203,228){\vector(0,-1){28}} \put(169,228){\vector(0,-1){28}}
\put(129,228){\vector(0,-1){28}} \put(89,228){\vector(0,-1){28}}
\put(49,228){\vector(0,-1){28}}
\put(205,201){\vector(4,3){37}} \put(245,201){\vector(4,3){37}}
\put(285,201){\vector(4,3){37}} \put(325,201){\vector(4,3){37}}

% arrows between 1 row and 2 row
\put(207,193){\oval(75,20)[b]} \put(244,194){\vector(0,1){0}}
\put(85,193){\vector(3,-1){79}} \put(125,193){\vector(3,-1){79}}
 \put(45,193){\vector(3,-1){79}} \put(287,167){\vector(3,1){79}} 
\put(207,167){\vector(3,1){79}} \put(247,167){\vector(3,1){79}} 

% arrows between 2 row and 3 row
\put(207,158){\oval(75,20)[b]} \put(244,159){\vector(0,1){0}}
\put(85,158){\vector(3,-1){79}} \put(125,158){\vector(3,-1){79}}
\put(45,158){\vector(3,-1){79}} \put(287,132){\vector(3,1){79}}
\put(207,132){\vector(3,1){79}} \put(247,132){\vector(3,1){79}} 

% arrows between K-1 row and K row
\put(207,78){\oval(75,20)[b]} \put(244,79){\vector(0,1){0}}
\put(85,78){\vector(3,-1){79}} \put(125,78){\vector(3,-1){79}}
\put(45,78){\vector(3,-1){79}} \put(287,52){\vector(3,1){79}}
\put(207,52){\vector(3,1){79}} \put(247,52){\vector(3,1){79}} 

% ovals below K row
\put(207,43){\oval(75,20)[b]} \put(244,44){\vector(0,1){0}}
\put(207,43){\oval(155,25)[b]} \put(284,44){\vector(0,1){0}}
\put(207,43){\oval(235,30)[b]} \put(324,44){\vector(0,1){0}}
\put(207,43){\oval(315,35)[b]} \put(364,44){\vector(0,1){0}}
\end{tiny}
\end{picture}

\vspace{.1in}

\begin{picture}(390,300)(0,0)
\put(115,10){{\bf Figure 4.} Graph of $g$ in Lemma~\ref{open h 2}.}

\begin{tiny}
% -1 row
\put(232,265){$(-1,1)$} \put(272,265){$(-1,2)$} 
\put(312,265){$(-1,3)$} \put(352,265){$(-1,4)\ \dots$}

% 0 row
\put(195,230){$(0,0)$} \put(235,230){$(0,1)$}
\put(275,230){$(0,2)$}
\put(155,230){$(0,-1)$} \put(115,230){$(0,-2)$} 
\put(75,230){$(0,-3)$} \put(35,230){$(0,-4)$} 
\put(23,230){$\dots$} \put(315,230){$(0,3)$}
\put(355,230){$(0,4)\ \dots$}

% 1 row
\put(195,195){$(1,0)$} \put(235,195){$(1,1)$}
\put(275,195){$(1,2)$}
\put(155,195){$(1,-1)$} \put(115,195){$(1,-2)$} 
\put(75,195){$(1,-3)$} \put(35,195){$(1,-4)$} 
\put(23,195){$\dots$} \put(315,195){$(1,3)$}
\put(355,195){$(1,4)\ \dots$}

% 2 row
\put(195,160){$(2,0)$} \put(235,160){$(2,1)$}
\put(275,160){$(2,2)$} 
\put(155,160){$(2,-1)$} \put(115,160){$(2,-2)$} 
\put(75,160){$(2,-3)$} \put(35,160){$(2,-4)$} 
\put(23,160){$\dots$} \put(315,160){$(2,3)$}
\put(355,160){$(2,4)\ \dots$}

% 3 row
\put(195,125){$(3,0)$} \put(235,125){$(3,1)$}
\put(275,125){$(3,2)$}
\put(155,125){$(3,-1)$} \put(115,125){$(3,-2)$} 
\put(75,125){$(3,-3)$} \put(35,125){$(3,-4)$} 
\put(23,125){$\dots$} \put(315,125){$(3,3)$}
\put(355,125){$(3,4)\ \dots$}

% dots
\put(203,100){$\vdots$} \put(243,100){$\vdots$}
\put(283,100){$\vdots$} \put(323,100){$\vdots$}
\put(163,100){$\vdots$} \put(123,100){$\vdots$} 
\put(83,100){$\vdots$} \put(43,100){$\vdots$}
\put(363,100){$\vdots$}

% K-2 row
\put(190,80){$(K-2,0)$} \put(230,80){$(K-2,1)$}
\put(270,80){$(K-2,2)$}
\put(150,80){$(K-2,-1)$}
\put(110,80){$(K-2,-2)$} \put(70,80){$(K-2,-3)$}
\put(30,80){$(K-2,-4)$} 
\put(17,80){$\dots$} \put(310,80){$(K-2,3)$}
\put(350,80){$(K-2,4)\ \dots$}

% K-1 row
\put(190,45){$(K-1,0)$} \put(230,45){$(K-1,1)$}
\put(270,45){$(K-1,2)$}
\put(150,45){$(K-1,-1)$}
\put(110,45){$(K-1,-2)$} \put(70,45){$(K-1,-3)$}
\put(30,45){$(K-1,-4)$} 
\put(17,45){$\dots$} \put(310,45){$(K-1,3)$}
\put(350,45){$(K-1,4)\ \dots$}

% arrows between -1 row and 0 row
\put(242,263){\vector(0,-1){28}} \put(322,263){\vector(0,-1){28}}
\put(285,263){\vector(-3,-1){79}} \put(365,263){\vector(-3,-1){79}}
\put(222,270){\oval(40,20)[t]} \put(202,270){\line(0,-1){33}} \put(242,269){\vector(0,-1){0}}
\put(242,270){\oval(160,40)[t]} \put(162,270){\line(0,-1){33}} \put(322,269){\vector(0,-1){0}}
\put(202,270){\oval(160,30)[t]} \put(122,270){\line(0,-1){33}} \put(282,269){\vector(0,-1){0}}
\put(202,270){\oval(320,50)[t]} \put(42,270){\line(0,-1){33}} \put(362,269){\vector(0,-1){0}}
\put(232,270){\oval(300,60)[t]} \put(82,270){\line(0,-1){33}} \put(382,269){\vector(0,-1){0}}

% arrows between 0 row and 1 row
\put(242,228){\vector(0,-1){28}} \put(282,228){\vector(0,-1){28}}
\put(322,228){\vector(0,-1){28}} \put(362,228){\vector(0,-1){28}}
\put(81,200){\vector(-4,3){37}}
\put(201,200){\vector(-4,3){37}} \put(161,200){\vector(-4,3){37}} \put(121,200){\vector(-4,3){37}} 

% arrows between 1 row and 2 row
\put(222,193){\oval(40,20)[b]} \put(202,194){\vector(0,1){0}}
\put(201,165){\vector(-4,3){37}} \put(161,165){\vector(-4,3){37}} \put(121,165){\vector(-4,3){37}} \put(81,165){\vector(-4,3){37}} 
\put(362,193){\vector(-4,-3){37}}
\put(282,193){\vector(-4,-3){37}}
\put(322,193){\vector(-4,-3){37}}

% arrows between 2 row and 3 row
\put(222,158){\oval(40,20)[b]} \put(202,159){\vector(0,1){0}}
\put(201,130){\vector(-4,3){37}} \put(161,130){\vector(-4,3){37}} \put(121,130){\vector(-4,3){37}} \put(81,130){\vector(-4,3){37}} 
\put(362,158){\vector(-4,-3){37}}
\put(282,158){\vector(-4,-3){37}}
\put(322,158){\vector(-4,-3){37}}

% arrows between K-1 row and K row
\put(222,78){\oval(40,20)[b]} \put(202,79){\vector(0,1){0}}
\put(201,50){\vector(-4,3){37}} \put(161,50){\vector(-4,3){37}} 
\put(121,50){\vector(-4,3){37}} \put(81,50){\vector(-4,3){37}}
\put(362,78){\vector(-4,-3){37}}
\put(282,78){\vector(-4,-3){37}}
\put(322,78){\vector(-4,-3){37}}

% ovals below K row
\put(222,43){\oval(40,20)[b]} \put(202,44){\vector(0,1){0}}
\put(222,43){\oval(120,25)[b]} \put(162,44){\vector(0,1){0}}
\put(222,43){\oval(200,30)[b]} \put(122,44){\vector(0,1){0}}
\put(222,43){\oval(280,35)[b]} \put(82,44){\vector(0,1){0}}
\end{tiny}
\end{picture}

\vspace{.1in}

\noindent
Department of Mathematics \newline
Ben Gurion University \newline
Beer Sheva, 84105 \newline
Israel \newline

\noindent Email: {\tt mesyan@bgu.ac.il}

\end{document}